\documentclass[11pt,letterpaper]{article}
\usepackage[utf8]{inputenc}
\usepackage[margin=1in]{geometry}

\usepackage{amsmath, amsthm, amssymb}
\usepackage{MnSymbol} 
\usepackage{graphicx,color}
\usepackage[hyphens]{url}
\usepackage{dsfont}
\usepackage{booktabs}
\usepackage[normalem]{ulem}
\usepackage{mathtools}
\usepackage{nicefrac}
\usepackage{multirow}
\usepackage{algorithm}
\usepackage[noend]{algpseudocode}
\usepackage{tikz} 
\usepackage{soul} 
\usepackage[T1]{fontenc}
\usepackage{bbm}
\usepackage{blkarray}
\usepackage{cancel}
\usepackage{enumitem}
\usepackage{float}
\usepackage{mathrsfs}
\usepackage[framemethod=tikz]{mdframed}
\usepackage{microtype}
\usepackage{ragged2e}
\usepackage{sectsty}
\usepackage{thmtools}
\usepackage{thm-restate}
\usepackage[Symbolsmallscale]{upgreek}
\usepackage{array}
\usepackage{subcaption}
\usepackage{tabularx}
\usepackage{euscript}
\usetikzlibrary{cd}

\usepackage[
  backend=biber,
  natbib=true,
  giveninits=true,
  maxnames=10,
  style=alphabetic,
  sorting=nyt,
  sortcites=true
]{biblatex}
\usepackage{hyperref}
\usepackage[nameinlink,capitalize]{cleveref}
\hypersetup{citecolor=violet, colorlinks=true, linkcolor=blue}

\numberwithin{equation}{section}
\newtheorem{theorem}{Theorem}[section]
\newtheorem{cor}[theorem]{Corollary}
\newtheorem{lemma}[theorem]{Lemma}
\newtheorem{remark}[theorem]{Remark}

\newtheorem{defin}[theorem]{Definition}



\makeatletter
\patchcmd{\@counteralias}
{\@ifdefinable{c@#1}}
{\expandafter\@ifdefinable\csname c@#1\endcsname}
{}{}
\makeatother

\makeatletter
\newcommand{\opnorm}{\@ifstar\@opnorms\@opnorm}
\newcommand{\@opnorms}[1]{%
	\left|\mkern-1.5mu\left|\mkern-1.5mu\left|
	#1
	\right|\mkern-1.5mu\right|\mkern-1.5mu\right|
}
\newcommand{\@opnorm}[2][]{%
	\mathopen{#1|\mkern-1.5mu#1|\mkern-1.5mu#1|}
	#2
	\mathclose{#1|\mkern-1.5mu#1|\mkern-1.5mu#1|}
}
\makeatother

\newcommand{\bE}{\mathbb{E}}

\newcommand{\bN}{\mathbb{N}}

\newcommand{\bP}{\mathbb{P}}

\newcommand{\bR}{\mathbb{R}}

\newcommand{\cD}{\mathcal{D}}
\newcommand{\cE}{\mathcal{E}}

\newcommand{\cN}{\mathcal{N}}
\newcommand{\cO}{\mathcal{O}}

\newcommand{\sfR}{\mathsf{R}}
\newcommand{\fc}{\mathfrak{c}}

\newcommand{\dd}{\mathrm{d}}
\newcommand{\eps}{\varepsilon}
\newcommand{\RGO}{\text{RGO}}
\newcommand{\TV}{\mathsf{TV}}
\newcommand{\KL}{\mathsf{KL}}
\newcommand{\FI}{\mathsf{FI}}

\newcommand{\deq}{\coloneqq}
\newcommand{\comp}{\mathsf{c}}
\newcommand{\eu}[1]{\EuScript{#1}}
\newcommand{\T}{\mathsf{T}}
\DeclarePairedDelimiter\abs{\lvert}{\rvert}
\DeclarePairedDelimiter\norm{\lVert}{\rVert}

\DeclareMathOperator\one{\mathbbm{1}}
\DeclareMathOperator\prox{prox}
\DeclareMathOperator\dist{dist}
\DeclareMathOperator\var{var}
\DeclareMathOperator\ent{ent}
\DeclareMathOperator\PI{PI}
\DeclareMathOperator\LSI{LSI}

\makeatletter
\def\blfootnote{\gdef\@thefnmark{}\@footnotetext}
\makeatother

\allowdisplaybreaks

\title{A proximal gradient algorithm for composite log-concave sampling}
 	 \author{
         Linghai Liu \\
         \texttt{\small linghai.liu@yale.edu} 
		 \and
		 Sinho Chewi \\
		 \texttt{\small sinho.chewi@yale.edu}
     }  
    \date{Yale University}

\begin{document}

	\maketitle

    \begin{abstract}
        We propose an algorithm to sample from composite log-concave distributions over $\mathbb{R}^d$, i.e., densities of the form $\pi\propto e^{-f-g}$, assuming access to gradient evaluations of $\alpha_f$-convex $f$ and a restricted Gaussian oracle (RGO) for $\alpha_g$-convex $g$, where $\alpha_g \ge 0$. The latter requirement means that we can easily sample from the density $\text{RGO}_{g,h,y}(x) \propto \exp(-g(x) -\frac{1}{2h}||y-x||^2)$, which is the sampling analogue of the proximal operator for $g$. If $f + g$ is $\alpha$-convex and $f$ is $\beta$-smooth, our sampler achieves $\varepsilon$ error in total variation distance using $\widetilde{\mathcal O}(\kappa \sqrt d \log^4(1/\varepsilon))$ oracle calls where $\kappa := 1 \vee \beta/\alpha$, with $0\le \alpha_g \le \beta$ and $\alpha := \alpha_f + \alpha_g > 0$, which matches prior state-of-the-art results for the case $g=0$. We further extend our results to cases where (1) $\pi$ is non-log-concave but satisfies a Poincar\'e or log-Sobolev inequality, and (2) $f$ is non-smooth but Lipschitz.
    \end{abstract}
    
    \section{Introduction}

One of the cornerstones of convex optimization is the problem of composite optimization:
\[
\min_{x \in \bR^d} \, F(x)  \deq  f(x) + g(x) \, ,
\]
where $f$ is convex and smooth while $g$ is convex but possibly non-smooth. The proximal gradient method approaches this problem by the iteration \cite{parikh2014proximal}: 
\[
x^{k+1} = \prox_{h g}(x^k - h \nabla f(x^k)) \,, \qquad \prox_{hg}(x)  \deq  \arg\min_{y \in \bR^d}{\Bigl\{ g(y) + \frac{1}{2h}\, \norm{y-x}^2 \Bigr\}} \,.
\]
The smooth part $f$ admits a gradient step, whereas the non-smooth part $g$ is handled by the proximal map $\prox_{hg}$. If $g = \iota_{\eu C}$ is the convex indicator of a closed convex set $\eu C$, then the proximal step is the Euclidean projection onto $\eu C$, and the method becomes projected gradient descent; if $g = \lambda\,\norm{\cdot}_1$, then the method becomes the celebrated iterative shrinkage thresholding algorithm (ISTA)~\cite{beck2009fast}.

In our work, we are interested in the sampling counterpart of composite optimization, called composite log-concave sampling. Here, the goal is to sample from a distribution $\pi$ on $\bR^d$ with density
\[
\pi(x)\propto \exp(-F(x)) \deq \exp(-f(x)-g(x))\, .
\]
Therefore, it is natural to ask:
\begin{center}
\emph{What is the analogue of the proximal gradient algorithm
for composite log-concave sampling?}
\end{center}
Previous studies on log-concave sampling mainly focused on the (non-composite) case $g = 0$ with methods such as the unadjusted Langevin algorithm (ULA) \cite{dalalyan2017theoretical, durmus2017nonasymptotic, durmus2019high, vempala2019rapid, chewi2025analysis} and the Metropolis-adjusted Langevin algorithm (MALA) \cite{dwivedi2019log, chewi2021optimal, wu2022minimax}. In contrast, our goal is to incorporate the non-smooth component $g$ directly.

To deal with non-trivial $g$, a natural first approach is to smooth $g$ by replacing it with its Moreau--Yosida envelope
\[
g^h(y) \deq \inf_{x \in \bR^d}{\Bigl\{ g(x) + \frac{1}{2h}\, \norm{y-x}^2 \Bigr\}} \, ,
\]
where $h > 0$, so that $f+g^h$ is smooth and Langevin algorithms can be applied to the smoothed potential. This idea underlies the Moreau--Yosida regularized unadjusted Langevin algorithm (MYULA) \cite{pereyra2016proximal, brosse2017sampling, bernton2018langevin, durmus2018efficient}. However, this strategy converges to a biased limit $\pi^h$ and suffers from poor dimensional dependence, especially in the case when $g = \iota_{\eu C}$ where the original MYULA analysis gives an $\widetilde{\cO}(d^5)$ bound~\cite[Theorem~2 and Table~1]{brosse2017sampling}. A related approach is the projected Langevin algorithm, which projects each noisy gradient update onto $\eu C$ \cite{bubeck2018sampling}. 

A more direct approach is the proximal sampler introduced in \cite{LeeSheTia21RGO}. Their seminal work puts forth that the sampling analogue of the proximal map of $g$ is the restricted Gaussian oracle (RGO), which returns an exact sample from a Gaussian-tilted version of $g$:
\begin{align*}
    \text{RGO}_{g,h,y}(x) \propto \exp\Bigl(-g(x) - \frac{1}{2h}\,\|y-x\|^2\Bigr)\,.
\end{align*}
A similar idea appeared in the earlier work of~\cite{mou2022efficient}, which used the RGO as a proposal distribution inside a Metropolis--Hastings algorithm; however, their algorithm requires computation of the normalizing constant of $\text{RGO}_{g,h,y}$ in addition to a sample. In \cite{LeeSheTia21RGO}, the proximal reduction framework considers the lifted distribution
\[
\tilde \pi (x,y) \propto \exp\Bigl( -f(x)-g(y) - \frac{1}{2h}\,\norm{x-y}^2 \Bigr)\,.
\]
Their algorithm applies Gibbs sampling to $\tilde\pi$, which amounts to alternating between the RGOs for $f$ and $g$, where the former is implemented via rejection sampling.
After running Gibbs sampling and keeping the $y$ sample, they still need to correct the resulting $y$-marginal $\tilde \pi^Y$ to the original target distribution $\pi$ by another rejection sampling step. Both algorithms in~\cite{mou2022efficient, LeeSheTia21RGO} achieve a complexity which scales polylogarithmically in $1/\eps$, where $\eps$ is the target accuracy, but incur $\widetilde{\cO}(d)$ dimension dependence.

More recently, works in the non-composite setting, i.e., $g=0$ and $\pi \propto e^{-f}$, have developed faster RGO implementations which only incur dimension dependence $\widetilde{\cO}(\sqrt d)$~\cite{fan2023improved, altschuler2024faster}.
In particular, the rejection sampling schemes developed in~\cite{fan2023improved} also provide RGO implementations in the composite setting, but they require semi-smoothness assumptions on $g$ (e.g., $g$ is Lipschitz). This makes them unusable for our goal of making no assumptions on $g$, aside from convexity, in order to accommodate examples such as the convex indicator $\iota_{\eu C}$.

\subsection{Contributions}
In this paper, we propose an algorithm which is arguably closer to a true proximal gradient algorithm for composite log-concave sampling. Specifically, we consider applying the proximal sampler algorithm of~\cite{LeeSheTia21RGO}, except that we lift to the joint density 
\[
\pi(x,y) \propto \exp \Bigl( - f(x) - g(x) - \frac{1}{2h}\, \norm{x-y}^2 \Bigr) \,.
\]
Again, we perform Gibbs sampling, and this time the $x$-marginal is 
precisely our target density $\pi \propto e^{-f-g}$, i.e., the algorithm is asymptotically unbiased.
The non-trivial Gibbs sampling step requires sampling from the density
\begin{align*}
    \text{RGO}_{f+g,h,y}(x) \propto \exp\Bigl(-f(x)-g(x) -\frac{1}{2h}\,\|y-x\|^2\Bigr)\,.
\end{align*}
Here, our idea is to use as a proposal
\begin{align*}
    \text{RGO}_{g,h,y-h\nabla f(y)}(x) \propto\exp\Bigl(-f(y)-\langle\nabla f(y),x-y\rangle -g(x) -\frac{1}{2h}\,\|y-x\|^2\Bigr)\,.
\end{align*}
Then, we apply a Metropolis--Hastings correction step.
Note that our proposal directly takes inspiration from the proximal gradient algorithm by linearizing the smooth term, and our resulting algorithm uses the RGO for $g$ alone.

Our main result shows that our method yields an $\eps$-accurate sample in total variation distance to the target $\pi$ in $\widetilde{\cO}(\kappa \sqrt d \log^4(1/\eps))$ gradient evaluations, where $\kappa$ is the condition number (smoothness of $f$ divided by the strong convexity of $f+g$).
To our knowledge, this is the first high-accuracy sampler for composite log-concave sampling with $\sqrt d$ dimension dependence given an RGO for arbitrary convex $g$ without \textit{any} smoothness assumptions. This result also matches the dimensional dependence of state-of-the-art results for the non-composite case $g=0$. Table~\ref{tab:composite-comparison} summarizes representative high-accuracy guarantees for composite sampling.

\begin{table}[htbp]
    \centering
    \caption{Comparison of high-accuracy guarantees for composite sampling.}
    \label{tab:composite-comparison}

    \small
    \setlength{\tabcolsep}{4pt}
    \renewcommand{\arraystretch}{1.2}

    \begin{tabularx}{\linewidth}{
        @{}ll>{\raggedright\arraybackslash}Xcc@{}
    }
        \toprule
        Method & Main oracles & Assumption on $g$
        & Dimension & Accuracy \\
        \midrule

        Lee et al.~\cite{LeeSheTia21RGO}
        & $\nabla f,\mathrm{RGO}_g$
        & Convex
        & $d$
        & $\log^{\mathcal O(1)}(1/\varepsilon)$ \\

        Mou et al.~\cite{mou2022efficient}
        & $\nabla f,\mathrm{RGO}_g,Z_g$
        & Convex, Lipschitz
        & $d$
        & $\log^{\mathcal O(1)}(1/\varepsilon)$ \\

        Fan et al.~\cite{fan2023improved}
        & $f+g$, $\operatorname{prox}_{h(f+g)}$
        & Convex, semi-smooth
        & $\sqrt{d}$
        & $\log^{\mathcal O(1)}(1/\varepsilon)$ \\

        \textbf{Ours}
        & $\nabla f,\mathrm{RGO}_g$
        & Convex
        & $\sqrt{d}$
        & $\log^4(1/\varepsilon)$ \\

        \bottomrule
    \end{tabularx}
\end{table}

Along the way, we provide an analysis of the independent Metropolis--Hastings algorithm (i.e., Metropolis--Hastings in which the proposal does not depend on the current state) based on easily checkable R\'enyi divergence conditions, which could be of independent interest.

Finally, we extend our results to other standard settings, allowing for non-convex $f+g$ or non-smooth $f$, and we provide experimental validation in the supplementary to demonstrate that the proposed algorithm is indeed implementable and samples reasonably well.

\subsection{Other related works}

\paragraph{Composite sampling}
Beyond the approaches discussed in the introduction, other works on composite log-concave sampling with a non-smooth component $g$ include~\cite{DurMajMia19LMCCvxOpt, SalKovRic19ProxLangevin, SalRic20ProxLangevin, HabHolPoc24Subgrad}. These works differ from ours in terms of the assumptions on the oracle. Specifically, ~\cite{SalKovRic19ProxLangevin, SalRic20ProxLangevin} assume access to $\prox_{hg}$ instead of the RGO for $g$, while~\cite{DurMajMia19LMCCvxOpt, HabHolPoc24Subgrad} study subgradient-based algorithms. Moreover, these methods use ULA-type discretizations or splittings with a constant step size, which incur dimensional dependence $\widetilde \cO(d)$ and polynomial dependence on $1/\eps$, as compared to high-accuracy samplers such as proximal sampler with $\text{polylog}(1/\eps)$ complexity dependence on the target accuracy $\eps$.

\paragraph{Proximal sampler}
As discussed above,~\cite{mou2022efficient} proposed an algorithm based on the Metropolis--Hastings correction of the RGO, and the RGO was later incorporated into the proximal sampler framework in \cite{LeeSheTia21RGO}. The work \cite{chen2022improved} established convergence results for the ideal proximal sampler under weak log-concavity or functional inequalities, and we leverage these results in our analysis. Both papers~\cite{LeeSheTia21RGO, chen2022improved} considered simple rejection sampling implementations of the RGO, which incur dimension dependence $\widetilde{\cO}(d)$, and which were extended to composite/semi-smooth settings~\cite{LiaChe23Prox, Yuan+23Networks}.
Subsequently,~\cite{fan2023improved, altschuler2024faster} developed RGO implementations with dimension dependence $\widetilde{\cO}(\sqrt d)$, with the former paper working in the composite setting but in which $g$ is semi-smooth. More recently, \cite{chen2026high} showed that the RGO implementation does not require function evaluations (only gradient evaluations), and~\cite{Chen+26HighAccStoch} considered the role of stochastic gradient noise.

\paragraph{Sampling from convex bodies} Sampling from a convex body $\eu C \subseteq \bR^d$ is a special case of composite log-concave sampling where we set $f \equiv 0$ and $g = \iota_{\eu C}$. The RGO then reduces to sampling from a Gaussian restricted to $\eu C$. Using this RGO, \cite{kook2024and} pioneered the use of the proximal sampler for sampling convex bodies. A line of works~\cite{kook2024and, Kook25Zeroth, KooVem25Integration, kook2025faster, kook2025renyi} then studied the complexity of implementing the RGO using membership queries to $\eu C$.
Also, \cite{dang2025oracle} studied RGO implementation given access to either a projection oracle or a separation oracle for $\eu C$.
The focus of our work is different, as we assume an RGO for $\eu C$ and show how this leads to an improved algorithm for composite log-concave sampling.

\section{Preliminaries}
\subsection{Divergences between probability measures}
We recall the divergences used throughout the paper to quantify distances between probability measures. 
\begin{defin}[R\'enyi divergence]
    \label{def:renyi-divergence}
    For two probability measures $\mu \ll \nu$ on $\bR^d$, the $p$-th order R\'enyi divergence ($p > 1$) is 
    \[
    \sfR_p(\mu \| \nu) = \frac{1}{p-1} \log \int \bigl(\frac{\dd\mu}{\dd\nu}\bigr)^p\, \dd \nu\,.
    \]
    Note that the R\'enyi divergence is also closely related to the $\KL$ and $\chi^2$ divergences: as $p \to 1$ we have $\sfR_1(\mu \| \nu)  \deq \KL(\mu \| \nu)$; for $p=2$ we have $\sfR_2(\mu \| \nu) = \log (1 + \chi^2(\mu \| \nu) )$, where 
    \[
    \KL (\mu\| \nu) = \int \frac{\dd \mu}{\dd \nu} \log \frac{\dd \mu}{\dd \nu} \dd \nu\,, \qquad \chi^2 (\mu \| \nu) = \int \bigl( \frac{\dd \mu}{\dd \nu}-1 \bigr)^2 \dd\nu = \int \bigl( \frac{\dd \mu}{\dd \nu} \bigr)^2 \dd\nu -1\,.
    \]
\end{defin}

\subsection{Functional inequalities}
We recall the definitions of the functional inequalities that we use.
\begin{defin}[Poincar\'e inequality]
    \label{def:poincare-inequality}
    We say that $\pi$ satisfies a Poincar\'e inequality (PI) with constant $1/\alpha$ if for all compactly supported and smooth test functions $\phi : \bR^d\to\bR$,
    \begin{align*}
        \var_\pi(\phi) \le \frac{1}{\alpha}\,\bE_\pi\bigl[\|\nabla\phi\|^2\bigr]\,.
    \end{align*}
\end{defin}
\begin{defin}[Log-Sobolev inequality]
    \label{def:log-sobolev-inequality}
    We say that $\pi$ satisfies a log-Sobolev inequality (LSI) with constant $1/\alpha$ if for all compactly supported and smooth test functions $\phi : \bR^d\to\bR$,
    \[
    \ent_\pi(\phi^2) \deq \bE_\pi \bigl[ \phi^2 \log \phi^2 \bigr] - \bE_\pi\bigl[\phi^2\bigr]\log \bE_\pi\bigl[\phi^2\bigr] \le \frac{2}{\alpha}\,\bE_\pi\bigl[\|\nabla\phi\|^2\bigr]\,.
    \]
\end{defin}
Under our notation, an LSI implies a PI with the same constant. In particular, by the Bakry--\'Emery criterion \cite{Bakry1985}, any $\alpha$-strongly log-concave distribution satisfies an LSI, and hence also a PI, with the same constant $1/\alpha$.

\subsection{Proximal sampler and restricted Gaussian oracle}
Suppose our target distribution on $\bR^d$ has density $\pi \propto \exp(-F)$, where $F: \bR^d \to \bR \cup \{\infty\}$. Fix a step size $h >0$ and extend the distribution to $\bR^d \times \bR^d$ as: 
\[
\pi(x,y) \propto \exp\Bigl(-F(x) - \frac{1}{2h}\, \norm{x-y}^2 \Bigr)\,.
\]
The proximal sampler runs Gibbs sampling on this augmented distribution. 
\begin{defin}[Proximal sampler]
    \label{def:proximal-sampler}
    Initialize $X_0 \sim \mu_0$. For $k = 0, 1, 2, \ldots$:
    \begin{itemize}
        \item Sample $Y_k \sim \pi^{Y \mid X = X_k}(y) = \cN(X_k, hI_d) \propto \exp\left( -\frac{1}{2h}\, \norm{y - X_k}^2\right)$.
        \item Sample $X_{k+1} \sim \pi^{X \mid Y = Y_k}(x) \propto \exp\left( -F(x)-\frac{1}{2h}\, \norm{x - Y_k}^2 \right)$.
    \end{itemize}
\end{defin}
As we can see from the definition, in order to implement the proximal sampler, we need to be able to sample from both conditional distributions, specifically from $\pi^{X \mid Y= Y_k}$, whose density is a Gaussian tilt centered at $Y_k$ of the target distribution $\pi$. This idea is formulated in the following definition of the restricted Gaussian oracle \cite{LeeSheTia21RGO}.

\begin{defin}[Restricted Gaussian oracle]
    \label{def:RGO}
    A restricted Gaussian oracle for a convex function $\varphi:\bR^d \to \bR \cup \{\infty\}$ with parameter $h > 0$ and center $v \in \bR^d$ ($\RGO_{\varphi, h,v}$) returns a sample from the density proportional to
    \[
    \exp\Bigl(-\varphi(x) - \frac{1}{2h}\,\|x - v\|^2\Bigr)\,.
    \]
\end{defin}
The RGO is exactly the oracle needed to implement the second update of the proximal sampler with $\varphi = F$ and $v = Y_k$. In the composite setting $F = f + g$ of this paper, we do not assume access to the RGO for $F$. Instead, we only assume access to RGO for the non-smooth term $g$, together with gradient access to the smooth term $f$.

\section{Main results}
\begin{theorem}
    \label{thm:main-result}
    Let $\pi$ be a distribution on $\bR^d$ with density $\pi \propto e^{-f-g}$, where $f$ is $\alpha_f$-convex and $\beta$-smooth and $g$ is $\alpha_g$-strongly convex, with $\beta \ge \alpha_g \ge 0$. Let $\alpha  \deq  \alpha_f + \alpha_g > 0$, $\kappa \deq 1 \vee \beta/\alpha$, and $\eps \in(0,1)$. Assume access to $x_\ast = \arg\min \{f + g\}$, first-order oracle of $f$, and the RGO for $g$. Then, with a suitable choice of parameters, Algorithm~\ref{alg:outer-loop-Gibbs-Sampling} uses $\widetilde \cO(\kappa \sqrt{d} \log^4(1/\eps))$ evaluations of $f$, $\nabla f$, and the RGO for $g$, and achieves $\eps$ total variation distance to $\pi$. 
\end{theorem}
\begin{remark} We make several clarifications on the assumptions of our main theorem.
\begin{itemize}
    \item The parameter $\alpha_f$ is allowed to be negative, as long as $\alpha_f > -\beta$ and $\alpha > 0$. In addition, we do not assume smoothness for $g$, so $g$ can be non-smooth.
    \item We can shift $f$ and $g$ so that $x_\ast$ becomes their shared minimizer, similar to~\cite[Proposition 23]{LeeSheTia21RGO}. We cover this and the RGO of the shifted $g$ in \S~\ref{ssec:shared-reduction}. Throughout the paper, we implicitly assume that this preprocessing step has been done.
    \item When applying Algorithm~\ref{alg:RGO-implementation}, we only need to compute the densities of the proposal and of the target ($\mu_k$ and $\nu_k$ respectively, in Algorithm~\ref{alg:outer-loop-Gibbs-Sampling}) up to proportionality.
    In particular, we do not require the normalizing constant for the RGO for $g$, unlike~\cite{mou2022efficient}.
\end{itemize}
\end{remark}
The proof of this theorem is in \S~\ref{subsection:proof-of-main-theorem}. The key ingredients of our theorem are the proximal sampler (Definition \ref{def:proximal-sampler}), stated in Algorithm \ref{alg:outer-loop-Gibbs-Sampling}, and the implementation of the conditional distribution $\pi^{X \mid Y}$ using the independent Metropolis--Hastings algorithm (Algorithm \ref{alg:RGO-implementation}) elaborated in \S~\ref{sec:analysis-of-independent-metropolis-hastings}. We show that a larger choice of $h = \widetilde{\Theta}(1/(\beta \sqrt{d}))$ controls the R\'enyi divergence between the proposal and the target, which suffices for the analysis by Theorem~\ref{thm:indep_mh}. Since the iteration complexity of the proximal sampler typically scales as $\widetilde \cO(1/(\alpha h))$ (neglecting dependence on other parameters), our larger choice of $h$ improves the overall dependence on the dimension. 
\begin{algorithm}[ht]
    \caption{\textsc{CompositeSampler}$(f, g, h, K,\eps)$}\label{alg:outer-loop-Gibbs-Sampling}
    \begin{algorithmic}
        \Require $f$ $\alpha_f$-convex and $\beta$-smooth; $g$ $\alpha_g$-convex; step size $h > 0$; $K \in \bN^\ast$; $\eps > 0$.
        \State Initialize $x_0 \sim \rho_0 \propto \exp(-g(\cdot)-\frac{2\beta-\alpha_g}{2}\,\|\cdot-x_\ast\|^2) = \text{RGO}_{g,(2\beta-\alpha_g)^{-1},x_*}$.
        \For{$k = 0$ to $K-1$}
            \State Sample $y_{k} \sim \pi^{Y \mid X = x_k} = \cN(x_k, hI_d)$.
            \State Define $\nu_k(x) \deq \pi^{X \mid Y = y_k}(x) \propto \exp(-f(x)-g(x)-\frac{1}{2h} \norm{x-y_k}^2)$.
            \State Define $\mu_k(x) \deq \RGO_{g,h,y_k - h\nabla f(y_k)}(x) \propto \exp(-\langle \nabla f(y_k), x-y_k\rangle -g(x)- \frac{1}{2h}\norm{x-y_k}^2)$.
            \State Sample $x_{k+1}$ by \textsc{IndependentMetropolis}$(\nu_k, \mu_k, \frac{\eps}{6K})$.
        \EndFor
        \State \textbf{Return} $x_K$.
    \end{algorithmic}
\end{algorithm}

\paragraph{Extensions to other settings}
We also extend our results to other standard settings. Here $\LSI(1/\alpha)$ and $\PI(1/\alpha)$ denote the log-Sobolev and Poincar\'e inequalities with constant $1/\alpha$, respectively.

\begin{theorem}\label{thm:extensions}
    Let $\pi$ be a distribution on $\bR^d$ with density $\pi \propto e^{-f-g}$ and $\eps \in (0,1)$. Assume access to $x_\ast = \arg\min \{f + g\}$, $\prox_g$, $\prox_{(f+g)}$, and the RGO of $g$, where $g$ is convex. 
    Then, there are algorithms that sample from $\pi$ up to $\eps$ error in total variation distance, with complexities given in the following table (up to hidden logarithmic factors).
    \begin{align*}
        \begin{array}{cccc}
             &\pi~\text{satisfies}~\LSI(1/\alpha) & \pi~\text{satisfies}~\PI(1/\alpha) & \pi~\text{log-concave} \\[0.25em]
             f~\text{is}~\beta\text{-smooth} & \kappa \sqrt d \log^5 \frac{\KL}{\eps}  & \kappa\sqrt d \log^5 \frac{1 + \chi^2}{\eps} & \frac{\beta\sqrt d\,W_2^2}{\eps^2} \\[0.5em]
             f~\text{is}~L\text{-Lipschitz} & \frac{L^2}{\alpha} \log^5 \frac{\KL}{\eps} & \frac{L^2}{\alpha} \log^5 \frac{1 + \chi^2}{\eps} & \frac{L^2\,W_2^2}{\eps^2}
        \end{array}
    \end{align*}
    Here, we use the shorthand $\KL \deq \KL(\rho_0\|\pi)$, $\chi^2 \deq \chi^2(\rho_0\|\pi)$, $W_2^2 \deq W_2^2(\rho_0,\pi)$, $\kappa \deq 1 \vee \beta/\alpha$. Also, in the Lipschitz case, we assume that $f$ is continuously differentiable.
\end{theorem}
The proof of this theorem is in \S~\ref{subsection:proof-of-main-theorem}.

\section{Analysis of the independent Metropolis--Hastings algorithm}
\label{sec:analysis-of-independent-metropolis-hastings}
The analysis of our RGO implementation is based on an analysis of the independent Metropolis--Hastings (MH) chain, that is, the special case of the MH chain in which the proposal distribution does not depend on the current state.
Since this could be of independent interest, we isolate this part of our proof into the following general result.

\begin{theorem}\label{thm:indep_mh}
    There is an absolute constant $C > 0$ such that the following holds. Let $\eps > 0$ and assume that our target distribution $\pi$ satisfies the Cheeger isoperimetric inequality with some constant $\mathsf{Ch} > 0$ and that the independent proposal kernel $\mu$ satisfies 
    \[
    \sfR_q(\mu \| \pi) \le C_\sfR q^\gamma \quad \text{and} \quad \sfR_q(\pi \| \mu) \le C_\sfR q^\gamma
    \]
    for all $1 \le q \le C\log(1/\eps)$, an absolute constant $\gamma > 0$, and some constant $C_\sfR > 0$ such that $C_\sfR\ll (\log(1/\eps))^{-(\gamma+1)}$. Assume further that both $\mu$ and $\pi$ are atomless. Then, the independent Metropolis--Hastings algorithm returns a sample whose marginal distribution is $\eps$-close to $\pi$ in total variation distance in $N = \cO(\log(1/\eps))$ iterations.
\end{theorem}
\begin{algorithm}[ht]
    \caption{\textsc{IndependentMetropolis}$(\pi, \mu, \eps)$}
    \label{alg:RGO-implementation}
    \begin{algorithmic}
        \Require Target density $\pi$; proposal $\mu$; $\eps > 0$.
        \State Define $N = \cO(\log (1/\eps))$.
        \State Initialize $x_0 \sim \mu$.
        \For{$n=0$ to $N-1$}
            \State Sample $U \sim \mathrm{Unif}(0,1)$. 
                \State Sample $z \sim \mu$.
                \State Compute $\alpha(x_n,z) = \frac{1}{2}\min \bigl(1,\, \frac{\pi(z)\, \mu(x_n)}{\pi(x_n)\, \mu(z)}\bigr)$.
                \If{$U \le \alpha(x_n,z)$}
                    \State $x_{n+1} \gets z$.
            \Else
                \State $x_{n+1} \gets x_n$.
            \EndIf
        \EndFor
        \State \textbf{Return} $x_N$.
    \end{algorithmic}
\end{algorithm}
The proof is in \S~\ref{app:proof-of-indep-mh}. The crucial part of the analysis is the $s$-conductance $\fc_s$ of the MH chain with the proposal $\mu$ independent of the current state. By using the $\frac{1}{2}$-lazy version of the MH chain, as reflected in the algorithm, we show that $\fc_s \gtrsim 1$. As the target distribution satisfies the Cheeger isoperimetric inequality, together with control of the R\'enyi divergence, we obtain the dimension-free complexity.

\section{Proof of the main theorem}

\subsection{Preliminaries}\label{ssec:prelim}
We record a few standard lemmas which we use in our arguments.
\begin{lemma}[{Herbst's argument,~Ledoux~\cite[\S 2.3]{ledoux2006concentration}}]
    \label{lem:LSI-herbst-subgaussian}
    Suppose $\pi$ is $\alpha$-strongly log-concave and $X \sim \pi$. Then for $F$ that is $L$-Lipschitz, $F(X) - \bE_\pi[F(X)]$ is $L^2/\alpha$-sub-Gaussian.
\end{lemma}
\begin{lemma}[{Basic lemma,~Chewi~\cite[Ch.~4]{chewi2026log}}]\label{lem:basic-lemma}
    Let $\pi \propto \exp(-V)$, with $V \succeq \alpha I_d \succ 0$.
    If $V$ is minimized at $x_\ast$, then $\bE_\pi[\|\cdot - x_\ast\|^2] \le d/\alpha$.
\end{lemma}
\begin{lemma}[{Change of measure,~Chewi~\cite[Ch.~6]{chewi2026log}}]\label{lem:change-of-measure-lemma}
    Suppose that a test function $\phi$ satisfies 
    \[
    \pi(\phi \ge \eta) \le \psi(\eta)\,, \quad \forall \eta > 0\,.
    \]
    Then, for any measure $\mu$ and $q > 1$, 
    \[
    \mu(\phi \ge \eta) \le \psi(\eta)^{\frac{q-1}{q}} \exp\Bigl( \frac{q-1}{q}\, \sfR_q(\mu \| \pi) \Bigr)\,.
    \]
\end{lemma}

\subsection{R\'enyi divergence lemma}

The log-Sobolev inequality bounds the KL divergence by the relative Fisher information, which is useful because the latter can be easier to bound. For instance, the relative Fisher information does not require computing normalizing constants.
The following lemma provides a R\'enyi divergence analogue of this bound, which is helpful for our subsequent arguments.
\begin{defin}[Relative Fisher information]
    For two probability measures $\mu \ll \nu$ on $\bR^d$, their relative Fisher information is defined as 
    \[
    \FI (\mu \| \nu) = \bE_\mu \Bigl[ \Bigl\| \nabla \log \frac{\mu}{\nu} \Bigr\|^2 \Bigr] \,.
    \]
\end{defin}
If $\nu$ satisfies $\LSI(1/\alpha)$, by definition of log-Sobolev inequality \ref{def:log-sobolev-inequality}, let $\phi = \sqrt{\mu/\nu}$, we have 
\[
\KL(\mu \| \nu) \le \frac{1}{2\alpha} \FI(\mu \| \nu)\,.
\]
\begin{lemma}[Upper bound for R\'enyi divergence]
    \label{lem:renyi-upper-bound}
    Let $\pi$ satisfy an LSI with constant $1/\alpha$. Then for any $q > 1$ and any $\lambda > q^2/(2\alpha)$, 
    \[
    \sfR_q(\mu \| \pi) \leq \frac{q^2/(2\alpha)}{\lambda - q^2/(2\alpha)} \log \bE_\pi \exp \Bigl( \lambda\, \Bigl\| \nabla \log \frac{\mu}{\pi}\Bigr\|^2 \Bigr)\,.
    \]
    Also, if $\lambda > q(q-1)/(2\alpha)$, 
    \[
    \sfR_q (\mu \| \pi) \leq \frac{q^2/(2\alpha)}{\lambda - q(q-1)/(2\alpha)} \log \bE_\mu \exp \Bigl( \lambda\, \Bigl\| \nabla \log \frac{\mu}{\pi}\Bigr\|^2 \Bigr)\,.
    \]
\end{lemma}
\begin{proof}
    As in Lemma 5 of \cite{vempala2019rapid}, by the log-Sobolev inequality, we have
    \[
    \sfR_q(\mu \| \pi) \le \frac{q^2}{2\alpha} \cdot \frac{\bE_\pi [ (\frac{\mu}{\pi})^q\, \| \nabla \log \frac{\mu}{\pi} \|^2]}{\bE_\pi[(\frac{\mu}{\pi})^q]}\,.
    \]
    Let $\gamma = (\frac{\mu}{\pi})^q \pi / \bE_\pi[(\frac{\mu}{\pi})^q]$. The above reads 
    \begin{equation}
        \label{eq:relative-Fisher-Renyi-LSI}
        \sfR_q (\mu \| \pi) \le \frac{q^2}{2\alpha}\, \bE_\gamma \Bigl[ \Bigl\| \nabla \log \frac{\mu}{\pi} \Bigr\|^2 \Bigr]\,.
    \end{equation}
    By Donsker--Varadhan duality, for any $\lambda > 0$, we have 
    \begin{equation}
        \label{eq:donsker-varadhan-variational}
        \bE_\gamma\Bigl[ \Bigl\| \nabla \log \frac{\mu}{\pi} \Bigr\|^2 \Bigr] \le \frac{1}{\lambda}\, \Bigl\{ \KL(\gamma \| \pi) + \log \bE_\pi \exp\Bigl( \lambda\, \Bigl\| \nabla \log \frac{\mu}{\pi} \Bigr\|^2 \Bigr) \Bigr\}\,.
    \end{equation}
    By applying the log-Sobolev inequality again, we have 
    \[
    \KL(\gamma \| \pi) \le \frac{1}{2\alpha}\, \FI(\gamma \| \pi) = \frac{q^2}{2\alpha}\, \bE_\gamma\Bigl[ \Bigl\| \nabla \log \frac{\mu}{\pi} \Bigr\|^2 \Bigr]\,.
    \]
    Substituting into \eqref{eq:donsker-varadhan-variational} and rearranging terms, we have
    \[
    \bE_\gamma \Bigl[ \Bigl\| \nabla \log \frac{\mu}{\pi} \Bigr\|^2 \Bigr] \le (\lambda - q^2/(2\alpha))^{-1} \log \bE_\pi  \exp\Bigl( \lambda\, \Bigl\| \nabla \log \frac{\mu}{\pi} \Bigr\|^2 \Bigr)\,.
    \]
    Thus, \eqref{eq:relative-Fisher-Renyi-LSI} becomes
    \[
    \sfR_q (\mu \| \pi) \le \frac{q^2/(2\alpha)}{\lambda - q^2/(2\alpha)} \log \bE_\pi \exp\Bigl( \lambda\, \Bigl\| \nabla \log \frac{\mu}{\pi} \Bigr\|^2 \Bigr)\,.
    \]
    The second statement follows similarly, except that we have
    \begin{align*}
    \KL(\gamma \| \mu)
    &\le \frac{q-1}{q}\, \KL(\gamma \| \pi)\,.
    \end{align*}
\end{proof}

\subsection{RGO implementation}
The key step in our analysis of Algorithm \ref{alg:outer-loop-Gibbs-Sampling} is to implement the conditional probability $\pi^{X \mid Y}$, the RGO of $f+g$. Suppose we are given $Y=y \in \bR^d$ and define
\begin{equation}
    \label{eq:subroutine-distribution}
    \nu(x)  \deq  \pi^{X \mid Y = y}(x) \propto \exp\Bigl( -f(x) - g(x) - \frac{1}{2h}\,\|x-y\|^2 \Bigr)\,,
\end{equation}
and our proposal 
\begin{equation}
    \label{eq:subroutine-proposal}
    \mu(x) \propto \exp \Bigl( - \langle \nabla f(y), x-y \rangle - g(x)- \frac{1}{2h}\, \|x-y\|^2 \Bigr)\,
\end{equation}
with mode 
\[
y^+ = \arg\min_{x \in \bR^d} \Bigl\{g(x)+\frac{1}{2h}\,\|x-(y-h\nabla f(y))\|^2\Bigr\} = \text{prox}_{hg}(y-h\nabla f(y))\,.
\]
The following key lemma controls the R\'enyi divergence between $\mu$ and $\nu$, which will later allow us to invoke Theorem~\ref{thm:indep_mh}.

\begin{lemma}
    Let $\mu$ and $\nu$ be those defined above in Eq.~\eqref{eq:subroutine-distribution} and Eq.~\eqref{eq:subroutine-proposal}. 
    Let $g$ be convex. 
    \begin{enumerate}
        \item Let $p > 1$. Assume that $f$ is $\alpha_f$-convex and $\beta$-smooth, and that $h \lesssim 1/(\beta p)$ for a sufficiently small implied constant. Then, 
        \begin{align}
        \sfR_p (\mu \| \nu) \vee \sfR_p(\nu\|\mu) \lesssim \beta^2 p^2 h\, \norm{y - y^+}^2 + \beta^2 p^2 d h^2 \,. \label{eq:renyi-divergence-bound-A3}
        \end{align}
        \item Let $p > 1$. Assume that $f$ is $L$-Lipschitz and continuously differentiable, and that $h\lesssim 1/L^2$. Then,
        \begin{align}\label{eq:renyi_bd_Lip}
        \sfR_p(\mu \| \nu) \vee \sfR_p(\nu\|\mu) \lesssim p^2 h L^2 \,.
        \end{align}
    \end{enumerate}
\end{lemma}
\begin{proof}
    In the first case, note that $\nu$ is $(1/h-\beta)$-strongly log-concave, which is at least $1/(2h)$ when $h \le 1/(2\beta)$.
    By Lemma~\ref{lem:renyi-upper-bound} with $\lambda = 2p^2 h$, 
    \begin{align}
        \label{eq:bounding-renyi-divergence}
        \sfR_p (\mu \| \nu) &\le \frac{p^2h}{2p^2h - p(p-1)h} \log \bE_\mu \exp \Bigl( 2p^2 h\, \Bigl\| \nabla \log \frac{\mu}{\nu} \Bigr\|^2 \Bigr) \notag \\
        &\le \log \bE_{x \sim \mu} \exp \bigl( 2p^2 h\, \| \nabla f(x)- \nabla f(y)\|^2 \bigr) \notag \\
        &\le \log \bE_{x \sim \mu} \exp \bigl( 2\beta^2 p^2 h\, \| x-y \|^2 \bigr) \qquad (\text{$f$ is $\beta$-smooth}) \notag \\
        &\le \log \bE_{x \sim \mu} \exp\bigl( 4 \beta^2 p^2 h\, (\|x-y^+\|^2 + \|y - y^+\|^2) \bigr) \notag \\
        &= 4 \beta^2 p^2 h \,\|y-y^+\|^2 + \log \bE_{x \sim \mu} \exp\bigl(4 \beta^2 p^2 h\, \|x-y^+\|^2 \bigr)\,,
    \end{align}
    where $y^+$ is the mode of $\mu$. Since $\mu$ is $1/h$-strongly log-concave and $\|\cdot - y^+\|$ is $1$-Lipschitz, by Herbst's argument (Lemma \ref{lem:LSI-herbst-subgaussian}), 
    \[
    Z = \|x-y^+\| - \bE_{x \sim \mu}\left[\|x-y^+\|\right] \quad \text{is $h$-sub-Gaussian\,.}
    \]
    With this, we can further bound the latter term in Eq.~\eqref{eq:bounding-renyi-divergence}:
    \[
    \|x-y^+\|^2 = \left( Z + \bE_{x\sim \mu}\left[ \|x-y^+\| \right] \right)^2 \notag \le 2 Z^2 + 2 \,\bE_{x\sim \mu}\left[ \|x-y^+\| \right]^2 \,,
    \]
    which implies 
    \[
    \log \bE_{x \sim \mu} \exp \bigl(4 \beta^2 p^2 h\, \|x-y^+\|^2\bigr) \le 8\beta^2 p^2 h\, \bE_{x \sim \mu}\left[ \|x-y^+\|^2 \right] + \log \bE_\mu \exp\bigl( 8\beta^2 p^2 h Z^2 \bigr) \,.
    \]
    For the first term, we apply Lemma~\ref{lem:basic-lemma} and obtain
    \[
    \bE_{x \sim \mu}\left[ \|x-y^+\|^2 \right] \le d h\,.
    \]
    For the second term, we use the equivalent definition of sub-Gaussian random variables~\cite[Proposition 2.5.2]{vershynin2018high}. Specifically, $h$-sub-Gaussianity of $Z$ implies 
    $\bE \exp( \tilde\lambda^2 Z^2 ) \le \exp(C^2 \tilde\lambda^2)$ with $|\tilde\lambda| \le \frac{1}{C}$ and $C \asymp \sqrt{h}$. Hence, 
    \[
    \log \bE_\mu \exp\bigl(8\beta^2 p^2 h Z^2 \bigr) \lesssim \beta^2 p^2 h^2\, ,
    \]
    where we used $h \lesssim 1/(\beta p)$ with a sufficiently small constant. Thus, Eq.~\eqref{eq:bounding-renyi-divergence} becomes: 
    \begin{equation*}
        \sfR_p(\mu \| \nu) \lesssim \beta^2 p^2 h\, \|y-y^+\|^2 + \beta^2 p^2 d h^2 + \beta^2 p^2 h^2 \lesssim \beta^2 p^2 h\, \|y-y^+\|^2 + \beta^2 p^2 d h^2\,.
    \end{equation*}
    The proof for the bound on $\sfR_p(\nu\|\mu)$ is the same as above: interchange the roles of $\mu$ and $\nu$, and use the first inequality of Lemma~\ref{lem:renyi-upper-bound} instead.
    
    For the second statement, since $f$ is differentiable and $L$-Lipschitz, $\norm{\nabla f} \le L$. Also, the negative log-density of $\nu$ is an $L$-Lipschitz perturbation of a $1/h$-strongly convex function, so $\nu$ satisfies a log-Sobolev inequality with constant $O(h)$ by~Brigati and Pedrotti~\cite{BriPed25HeatFlow}, provided that $h \lesssim 1/L^2$. We apply Lemma~\ref{lem:renyi-upper-bound} with $\lambda = Cp^2 h$ for a universal constant $C > 0$:
    \[
    \sfR_p(\mu \| \nu) \le \log \bE_{x \sim \mu} \exp\bigl(Cp^2 h\, \norm{\nabla f(x) - \nabla f(y)}^2 \bigr) \lesssim p^2 L^2 h \,.
    \]
    The reverse bound also follows easily.
\end{proof}

The first case of the lemma above suggests that we need a bound for $\|y-y^+\|^2$ along the iterates of Algorithm~\ref{alg:outer-loop-Gibbs-Sampling}, which is the main focus of \S~\ref{subsection:concentration-along-perfect-RGO}.

\subsection{Concentration along the proximal sampler}
\label{subsection:concentration-along-perfect-RGO}
In this subsection, we bound $\norm{y-y^+}^2$ along Algorithm \ref{alg:outer-loop-Gibbs-Sampling}, where at each step $k$, the conditional distribution $\pi^{X_{k+1} \mid Y_k}$ is exact.
In this subsection, we assume that $f$ is $\alpha_f$-convex and $\beta$-smooth, and that $g$ is $\alpha_g$-convex, where $\alpha \deq \alpha_f + \alpha_g > 0$, $\alpha_f > -\beta$, $0 \le \alpha_g \le \beta$.

Suppose $Y_k = y$ and recall that $y^+ = \text{prox}_{hg}(y - h \nabla f(y))$ is the mode of our proposal $\mu$. By first-step optimality,
\[
-\nabla f(y) -\frac{1}{h}\, (y^+ -y) \in \partial g(y^+)\,.
\]
By the definition of subgradient, $\forall p \in \partial g(y^+)$, we have 
\[
g(z) \ge g(y^+) + \langle p, z - y^+ \rangle\,, \qquad \forall z \in \bR^d\,. 
\]
We evaluate this expression at $z = y$ and $p = -\nabla f(y) - \frac{1}{h}\, (y^+ -y)$: 
\begin{align*}
    g(y) &\ge g(y^+) + \Bigl\langle -\nabla f(y) - \frac{1}{h}\,(y^+-y), \,y - y^+ \Bigr\rangle \\
    &= g(y^+) + \frac{1}{h}\, \|y-y^+\|^2 - \langle \nabla f(y), y-y^+ \rangle \\
    \Longrightarrow g(y^+) &\le g(y) + \langle \nabla f(y), y - y^+ \rangle - \frac{1}{h}\, \|y-y^+\|^2\,. 
\end{align*}
On the other hand, by $\beta$-smoothness of $f$, 
\[
f(y^+) \le f(y) + \langle \nabla f(y), y^+ - y \rangle + \frac{\beta}{2}\, \|y - y^+ \|^2\,.
\]
Add up the two inequalities and define $F  \deq  f + g$. We obtain
\[
F(y^+) \le F(y) - \bigl(\frac{1}{h}-\frac{\beta}{2}\bigr)\, \|y - y^+ \|^2\,.
\]
Since $h \le 1/\beta$, we have
\begin{equation}
    \label{eq:y-y+bounded-by-F}
    \|y - y^+ \|^2 \le 2h\,(F(y) - F(y^+) ) \le 2h\, (F(y) - F_\ast)\,,
\end{equation}
where $F_\ast \deq \inf F$.

Note that the argument above holds for any $y$. We first prove the concentration of $F-F_\ast$ under the stationary distribution $\pi \propto e^{-f-g}$. Then, we use Lemma~\ref{lem:change-of-measure-lemma} to establish concentration at the initial distribution for the $y$ variable. Finally, we apply the data-processing inequality to show that at each iteration $k$ in Algorithm~\ref{alg:outer-loop-Gibbs-Sampling}, we still have a suitable concentration bound. 

\begin{lemma}
    \label{lem:concentration-F-gap}
    Let $F$ be convex, $F_* \deq \inf F$, and $\pi\propto \exp(-F)$.
    \begin{enumerate}
        \item For all $\lambda \in (0,1)$,
        \begin{align*}
            \bE_\pi \exp(\lambda\,(F-F_*)) \le (1-\lambda)^{-d}\,.
        \end{align*}
        \item For all $\delta \in (0,1)$,
        \begin{align*}
            \pi\bigl(F-F_* \ge d + \sqrt{2d\log(1/\delta)} + \log(1/\delta)\bigr) \le \delta\,.
        \end{align*}
    \end{enumerate}
\end{lemma}
\begin{proof}
For the first statement,
\begin{align*}
    \bE_\pi[\exp(\lambda \,(F - F_\ast))] &= \frac{\int \exp\left( \lambda\,(F(x) - F_\ast) -F(x) \right) \dd x}{\int \exp(-F(x))\, \dd x} \\
    &= \frac{\int \exp(-(1-\lambda)\, (F(x) - F_\ast) )\, \dd x}{\int \exp(-(F(x) - F_\ast) )\, \dd x}\,.
\end{align*}
By convexity of $F$, 
\[
(1-\lambda) F(x) + \lambda F_\ast \ge F((1-\lambda) x + \lambda x_\ast)\,,
\]
which implies 
\[
(1-\lambda)\,(F(x) - F_\ast) \ge F((1-\lambda) x + \lambda x_\ast) - F_\ast\,.
\]
Thus, the ratio becomes
\begin{align*}
    \frac{\int \exp(-(1-\lambda)\, (F(x) - F_\ast) )\, \dd x}{\int \exp(-(F(x) - F_\ast) )\, \dd x} &\le \frac{\int \exp(-F((1-\lambda) x + \lambda x_\ast) + F_\ast )\, \dd x}{\int \exp(-(F(x) - F_\ast) )\, \dd x}  \\
    &= \frac{\int \exp(-F((1-\lambda) x + \lambda x_\ast) )\, \dd x}{\int \exp(-F(x))\, \dd x} \\
    &= \frac{\int \exp(-F(y))\, \abs{\det [(1-\lambda) I_d]}^{-1}\,\dd y}{\int \exp(-F(x))\, \dd x} \\
    &= (1-\lambda)^{-d}\,.
\end{align*}
Then,
\begin{align*}
    \log \bE_\pi\exp(\lambda\,(F-F_* - d))
    &\le d\log\frac{1}{1-\lambda} -d\lambda
    \le \frac{d\lambda^2}{2\,(1-\lambda)}\,.
\end{align*}
According to the discussion in~\cite[\S 2.4]{BouLugMas13Con}, $F-F_*-d$ has a sub-gamma right tail with variance factor $d$ and scale parameter $1$, and therefore satisfies the tail bound
\begin{align*}
    \pi(F-F_* - d \ge \sqrt{2d\log(1/\delta)} + \log(1/\delta)) \le \delta\,.
\end{align*}
This proves the second statement.
\end{proof}

By the AM--GM inequality, $\sqrt{2d \log(1/\delta)} \le \frac{1}{2}\,(2d + \log(1/\delta))$. Based on Lemma~\ref{lem:concentration-F-gap}, for $t > 0$ and $\zeta \in (0,1)$, letting $\log(1/\delta) = \log(1/\zeta) + t$, we obtain
\begin{align}
    \label{eq:concentration-at-stationarity}
    \pi\bigl(F - F_\ast \ge 2d + \frac{3}{2} \log(1/\zeta) + \frac{3}{2}\, t\bigr) &= \pi\bigl(F - F_\ast \ge 2d + \frac{3}{2} \log(1/\delta)\bigr) \notag \\
    &\le \pi\bigl(F - F_\ast \ge d + \sqrt{2d \log(1/\delta)} + \log (1/\delta) \bigr) \notag \\
    &\le \delta = \zeta e^{-t}\,.
\end{align}

We next show how to transfer this concentration bound to the iterates by change of measure, beginning with a bound on the R\'enyi divergence at initialization.

Recall that in Algorithm~\ref{alg:outer-loop-Gibbs-Sampling}, the initial distribution $\rho_0$ is defined as 
\[
\rho_0(x) \propto \exp\Bigl( -g(x) - \frac{2\beta-\alpha_g}{2}\,\|x-x_\ast\|^2 \Bigr)\,,
\]
where $x_\ast$ is the shared minimizer of $f$ and $g$.

\begin{lemma}\label{lem:initial-distribution-outer-loop}
    Assume that $f$ is $\alpha_f$-convex and $\beta$-smooth, and $g$ is $\alpha_g$-convex, where $\alpha_f > -\beta$, $0 \le \alpha_g \le \beta$, and $\alpha \deq \alpha_f + \alpha_g > 0$.
    Then, for $\kappa \deq 1 \vee \beta/\alpha$,
    \begin{align*}\label{eq:warm-start-for-outer-loop}
        \sfR_\infty(\rho_0 \|\pi) \le \frac{d}{2}\log(2\kappa)\,.
    \end{align*}
\end{lemma}
\begin{proof}
For any $x \in \bR^d$,
\begin{align*}
    \frac{\rho_0(x)}{\pi(x)} &= \frac{\exp(-g(x) - \frac{2\beta-\alpha_g}{2}\, \|x-x_\ast\|^2)}{\exp(-f(x)-g(x))} \cdot \frac{\int \exp(-f(\eta) -g(\eta) )\, \dd \eta}{\int \exp(-g(\eta) - \frac{2\beta-\alpha_g}{2}\, \|\eta-x_\ast\|^2)\, \dd \eta} \\
    &\le \frac{\exp(-\frac{2\beta-\alpha_g}{2}\,\|x-x_\ast\|^2)}{\exp(-f_\ast - \frac{\beta}{2}\,\|x-x_\ast\|^2)} \cdot \frac{\int \exp(-g(\eta) - f_\ast - \frac{\alpha_f}{2}\,\|\eta - x_\ast\|^2)\, \dd \eta}{\int \exp(-g(\eta) - \frac{2\beta-\alpha_g}{2}\,\|\eta - x_\ast\|^2)\, \dd \eta} \\
    &= \frac{\int \exp(-g(\eta) - \frac{\alpha_f}{2}\,\|\eta - x_\ast\|^2)\,\dd \eta}{\int \exp(-g(\eta) - \frac{2\beta-\alpha_g}{2}\,\|\eta - x_\ast\|^2)\,\dd \eta} \le \Bigl(1 + \frac{2\beta-\alpha_g-\alpha_f}{\alpha} \Bigr)^{d/2} = (2\kappa)^{d/2}\,,
\end{align*}
where the last inequality uses ~\cite[Proposition 61]{LeeSheTia21RGO} with $f = g + \frac{\alpha_f}{2}\,\|\cdot - x_\ast\|^2$ and $\frac{1}{\lambda} = 2\beta-\alpha_g$. 
\end{proof}
Now we apply Lemma~\ref{lem:change-of-measure-lemma} (change of measure) with $q=\infty$ and $\phi = F-F_\ast -2d-3\log(1/\zeta)$. By Eq.~\eqref{eq:concentration-at-stationarity}, we have:
\[
\rho_0 \Bigl(F - F_\ast -2d-\frac{3}{2}\log(1/\zeta) \ge \frac{3}{2}\,t\Bigr) \le \zeta \exp\Bigl(-t +\frac{d}{2}\log(2\kappa)\Bigr) \le \zeta\,,
\]
where we choose $t = \frac{d}{2} \log(2\kappa)$.

By the data-processing inequality, for each $k \in \bN^\ast$, we have 
\[
\sfR_\infty (\rho_k \| \pi) \le \sfR_\infty(\rho_0 \| \pi) \,.
\]
Therefore, by Lemma~\ref{lem:change-of-measure-lemma}, Eq.~\eqref{eq:concentration-at-stationarity}, Lemma~\ref{lem:initial-distribution-outer-loop}, and our choice of $t$,
\begin{align*}
    \rho_k\Bigl(F - F_\ast -2d - \frac{3}{2}\log(1/\zeta) \ge \frac{3}{2}\,t\Bigr) &\le \zeta \exp\Bigl(-t + \sfR_\infty(\rho_k \| \pi) \Bigr) \\
    &\le \zeta\exp\Bigl(-t + \sfR_\infty(\rho_0 \| \pi) \Bigr) \le \zeta\,.
\end{align*}
Denote $y_k \overset{\cD}{=} x_k + \sqrt{h}\xi$, where $\xi \sim \cN(0,I_d)$ is independent of $x_k$. Let $x_k^+ \deq \text{prox}_{hg}(x_k - h \nabla f(x_k) )$ and $y_k^+ \deq \text{prox}_{hg}(y_k - h \nabla f(y_k) )$. Since $\text{prox}_{hg}$ is non- and $h \le 1/\beta$, 
\begin{align*}
    \norm{x_k^+ - y_k^+} &= \norm{\prox_{hg}(x_k - h \nabla f(x_k)) - \prox_{hg}(y_k - h \nabla f(y_k))} \\
    &\leq \norm{x_k - h \nabla f(x_k) - (y_k - h \nabla f(y_k))} \\
    &\leq \norm{x_k - y_k} + h\norm{\nabla f(x_k) - \nabla f(y_k)} \\
    &\leq (1 + \beta h) \norm{x_k - y_k} \\
    &\leq 2 \norm{x_k - y_k} \overset{\cD}{=} 2 \sqrt{h}\|\xi\| \,,
\end{align*}
which implies
\begin{equation}
    \label{eq:y-y^+-high-prob-bound}
    \|y_k - y_k^+\|^2 \le  3\,(\|y_k-x_k\|^2 + \|x_k-x_k^+\|^2 + \|x_k^+ - y_k^+\|^2) \le 3\,\|x_k - x_k^+\|^2 + 15h\,\|\xi\|^2\,.
\end{equation}
By Gaussian concentration, we have that for any $t_1 = \sqrt{2\log (1/\zeta)} > 0$, 
\[
\bP(\|\xi\| \ge \sqrt{d} + t_1) \le \exp(-t_1^2/2) = \zeta\,.
\]
Combined with Eq.~\ref{eq:y-y+bounded-by-F} and Eq.~\ref{eq:y-y^+-high-prob-bound}, with probability $\ge 1 - 2\zeta$ under the $y$-marginal at the $k$-th iteration,
\begin{align}
    \|y_k - y_k^+\|^2 &\le 3 \cdot 2h\, \bigl(2d + \frac{3}{4}d\log (2\kappa) + \frac{3}{2}\log(1/\zeta)\bigr) + 15h \cdot \bigl( 2d + 4\log(1/\zeta)\bigr) \notag \\
    &= 42dh + \frac{9}{2}dh\log(2\kappa) + 69 h \log(1/\zeta)\,.\label{eq:bound-on-norm-square}
\end{align}

\subsection{Analysis using the perfect RGO}\label{subsection:analysis-along-perfect-RGO}

We recall the following results from~\cite{chen2022improved}.

\begin{lemma}[Convergence of the ideal proximal sampler]\label{lem:KL-convergence-outer-loop}
    Let $(\rho_k)_{k\in\bN}$ be the iterates of the proximal sampler with step size $h$ and stationary distribution $\pi$.
    \begin{enumerate}
        \item Let $\pi$ satisfy an LSI with constant $1/\alpha$. Then,
    \[
    \KL( \rho_k \| \pi) \le \frac{\KL(\rho_0 \| \pi)}{(1+\alpha h)^{2k}}\,.
    \]
    \item Let $\pi$ satisfy a PI with constant $1/\alpha$.
    Then,
    \begin{align*}
        \chi^2(\rho_k\|\pi) \le \frac{\chi^2(\rho_0\|\pi)}{(1+\alpha h)^{2k}}\,.
    \end{align*}
    \item Let $\pi$ be log-concave.
    Then,
    \begin{align*}
        \KL(\rho_k\|\pi) \le \frac{W_2^2(\rho_0,\pi)}{kh}\,.
    \end{align*}
    \end{enumerate}
\end{lemma}

\subsection{Error analysis for the proximal sampler}
\label{subsection:analysis-along-RGO-hat}
In this subsection, we denote $\rho_k$ and $\hat \rho_k$ as the marginal distributions for $x$ along the proximal sampler using the perfect RGO $\pi^{X \mid Y_k}$ and our implemented RGO $\hat \pi^{X \mid Y_k}$ at the $k$-th iteration in Algorithm \ref{alg:outer-loop-Gibbs-Sampling}, respectively. Note that $\rho_0 = \hat{\rho}_0$ as defined in the algorithm. We also denote the $y$-marginals along the perfect RGO and the implemented RGO at the $k$-th iteration as $\rho^Y_k$ and $\hat \rho^Y_k$, respectively. 

The tools that we use for the analysis are the results in the previous subsections that hold with high probability, as well as the following lemma on coupling: 
\begin{lemma}[{\cite[Proposition 4.7]{wilmer2009markov}}]
    \label{lem:coupling-TV}
    Let $\mu$ and $\nu$ be two probability distributions on $\Omega$. Then 
    \[
    \| \mu - \nu \|_\TV = \inf \{ \bP(X \neq Y) \mid (X, Y) \text{ is a coupling of $\mu$ and $\nu$} \}\,.
    \]
\end{lemma}
Let $\zeta \in (0,1)$. Suppose at the $k$-th iteration, we have iterates $x_k \sim \rho_k$ and $\hat x_k \sim \hat \rho_k$. By Lemma \ref{lem:coupling-TV}, there exists a coupling of $\rho_k$ and $\hat \rho_k$ for which $x_k = \hat x_k$ with probability $1 - \| \rho_k - \hat \rho_k \|_\TV$. Draw $\xi \sim \cN(0, I_d)$ independent of this coupling and set $y_k \deq x_k + \sqrt h\xi$, $\hat y_k \deq \hat x_k + \sqrt h \xi$. Then, we have $y_k = \hat y_k$ with the same probability under this coupling of $\rho_k^Y$ and $\hat \rho_k^Y$. 

At the $(k+1)$-th step, $x_{k+1} \sim \rho^Y_k \pi^{X \mid Y_k}$ and $\hat x_{k+1} \sim \hat \rho^Y_k \hat \pi^{X \mid Y_k}$.
We are able to show that $\pi^{X \mid Y_k=y}$ and our implementation $\hat \pi^{X \mid Y_k=y}$ are close with the following corollary.
\begin{cor}[RGO implementation]\label{cor:RGOhat-is-close}
It holds that
\[
\bigl\lVert \pi^{X \mid Y_k=y_k} - \hat \pi^{X \mid Y_k=y_k} \bigl\rVert_{\TV} \le \zeta \qquad \text{after } N=\cO(\log(1/\zeta)) \text{ iterations}
\]
in the following two cases.
\begin{enumerate}
    \item Let $f$ be $\alpha_f$-convex and $\beta$-smooth and $g$ be $\alpha_g$-convex, where $0 \le \alpha_g \le \beta$ and $\alpha \deq \alpha_f + \alpha_g > 0$. Here, we assume that $y_k \in \cE$, where $\cE$ is the event on which $\|y_k - y_k^+\|^2$ enjoys the concentration bound~\eqref{eq:bound-on-norm-square}, and that $h \asymp 1/\bigl(\beta \sqrt{d\log(2\kappa)}\,\log^2(1/\zeta)\bigr)$ for a sufficiently small constant.
    \item Let $f$ be continuously differentiable and $L$-Lipschitz, and let $g$ be convex.
    Here, we take $h\asymp 1/\bigl(L^2\log^3(1/\zeta)\bigr)$ for a sufficiently small constant.
\end{enumerate}
\end{cor}
\begin{proof}
    This corollary directly follows from Theorem~\ref{thm:indep_mh} on the independent Metropolis--Hastings algorithm. Let $\zeta \in (0,1)$. In the first case, since $y_k \in \cE$, the concentration bound in \S~\ref{subsection:concentration-along-perfect-RGO} reads
    \[
    \norm{y_k - y_k^+}^2 \lesssim dh \log(2\kappa) + h \log(1/\zeta) \,.
    \]
    Substituting into Eq.~\eqref{eq:renyi-divergence-bound-A3}, we have
    \begin{align*}
    \sfR_p(\mu \| \nu) \vee \sfR_p(\nu\|\mu)
    &\lesssim \beta^2 p^2 d h^2 + \beta^2 p^2 h \left(dh + dh \log (2\kappa) + h\log (1/\zeta)\right)\\
    &\lesssim \underbrace{\beta^2 d h^2 \log(2\kappa) \log(1/\zeta)}_{C_\sfR} \,\,p^2\,.
    \end{align*}
    By choosing $h \asymp 1/\bigl(\beta \sqrt{d\log(2\kappa)}\,\log^2(1/\zeta)\bigr)$ with a sufficiently small constant, we can apply Theorem~\ref{thm:indep_mh} with $\gamma=2$, thus proving the desired result.

    In the second case, we instead use Eq.~\eqref{eq:renyi_bd_Lip}, which states that
    \begin{align*}
        \sfR_p(\mu \| \nu) \vee \sfR_p(\nu\|\mu)
        &\lesssim \underbrace{hL^2}_{C_\sfR}\,p^2\,.
    \end{align*}
    Here, we choose $h\asymp 1/\bigl(L^2\log^3(1/\zeta)\bigr)$ and apply Theorem~\ref{thm:indep_mh} with $\gamma =2$.
\end{proof}

By the triangle inequality and the data-processing inequality,
\begin{align}
    \label{eq:TV-outer-loop-induction-step}
    \|\rho_{k+1} - \hat \rho_{k+1} \|_\TV &= \|\rho_k^Y \pi^{X \mid Y_k} - \hat \rho_k^Y \hat \pi^{X \mid Y_k} \|_\TV  \notag \\
    &\le \|\rho_k^Y \pi^{X \mid Y_k} - \rho_k^Y \hat \pi^{X \mid Y_k} \|_\TV + \|\rho_k^Y \hat \pi^{X \mid Y_k} - \hat \rho_k^Y \hat \pi^{X \mid Y_k} \|_\TV \notag \\
    &\le \|\rho_k^Y \pi^{X \mid Y_k} - \rho_k^Y \hat \pi^{X \mid Y_k} \|_\TV + \|\rho_k^Y - \hat \rho_k^Y \|_\TV \notag \\
    &\le \|\rho_k^Y \pi^{X \mid Y_k} - \rho_k^Y \hat \pi^{X \mid Y_k} \|_\TV + \|\rho_k - \hat \rho_k\|_\TV \,.
\end{align}
Given $Y_k = y$, if $y \in \cE$, draw $(x_{k+1}, \hat x_{k+1})$ from the maximal coupling as in Lemma \ref{lem:coupling-TV}; otherwise, draw $(x_{k+1}, \hat x_{k+1})$ from the independent coupling $\pi^{X \mid Y_k} \otimes \hat\pi^{X \mid Y_k}$. Then,
\begin{align*}
    \|\rho_k^Y \pi^{X \mid Y_k} - \rho_k^Y \hat \pi^{X \mid Y_k} \|_\TV &\le \| \text{Law}(y, x_{k+1}) - \text{Law} (y, \hat x_{k+1})\|_\TV \quad (\text{data-processing inequality})\\
    &\le \bP(x_{k+1} \neq \hat x_{k+1}) \quad (\text{by Lemma~\ref{lem:coupling-TV}})\\
    &= \bP(x_{k+1} \neq \hat x_{k+1},\, y \notin \cE) + \bP(x_{k+1} \neq \hat x_{k+1},\, y \in \cE) \\
    &\le \rho_k^Y(\cE^\comp) + \bP(x_{k+1} \ne \hat x_{k+1},\, y \in \cE) \\
    &\le 2\zeta + \zeta = 3\zeta\,,
\end{align*}
where for the first term we used $\rho_k^Y(\cE) \ge 1-2\zeta$ from \S~\ref{subsection:concentration-along-perfect-RGO} and for the second term we note that under our constructed coupling for $y \in \cE$, by Corollary~\ref{cor:RGOhat-is-close},
\[
\bP(x_{k+1} \neq \hat x_{k+1},\, y \in \cE) = \int_\cE \bigl\| \pi^{X \mid Y_k = y} - \hat \pi^{X \mid Y_k = y}\bigr\|_\TV\, \rho^Y_k(\dd y) \le \zeta\,.
\]
Therefore, the induction step in Eq.~\eqref{eq:TV-outer-loop-induction-step} becomes 
\[
\|\rho_{k+1} - \hat \rho_{k+1} \|_\TV \le 3\zeta + \|\rho_k - \hat \rho_k \|_\TV\,.
\]
Since $\rho_0 = \hat \rho_0$, by induction, we have
\[
\| \rho_K - \hat \rho_K \|_\TV \le 3K\zeta\,.
\]
\subsection{Proof of Theorems~\ref{thm:main-result} and~\ref{thm:extensions}}\label{subsection:proof-of-main-theorem}

\begin{proof}[Proof of Theorem~\ref{thm:main-result}]
    Let $\eps, \zeta \in (0,1)$, and $h \asymp 1/\bigl(\beta \sqrt{d \log (2\kappa)}\, \log^2(1/\zeta) \bigr)$ as in Corollary~\ref{cor:RGOhat-is-close}.
    Since $\pi \propto e^{-f-g}$ is $\alpha$-strongly log-concave, it satisfies an LSI with constant $1/\alpha$.
    Hence, by Pinsker's inequality, Lemma~\ref{lem:KL-convergence-outer-loop}, and monotonicity of R\'enyi divergence, we have 
    \begin{align}
        \| \rho_k - \pi \|_\TV \le \sqrt{\frac{1}{2}\, \KL (\rho_k \| \pi)} \le \sqrt{\frac{\KL(\rho_0 \|\pi)}{2\,(1+\alpha h)^{2k}}} \le \sqrt{\frac{\sfR_\infty (\rho_0 \| \pi)}{2\,(1+\alpha h)^{2k}}} \,.
    \end{align}
    By Lemma~\ref{lem:initial-distribution-outer-loop}, $\sfR_\infty(\rho_0 \| \pi) \le \frac{d}{2} \log(2\kappa)$, so $\| \rho_K - \pi \|_\TV \le \eps$ if $K \ge \frac{\log\frac{d\log(2\kappa)}{4\eps^2}}{2\log (1+\alpha h)}$.
    Since $\log (1 +x ) \ge x - x^2/2$ for all $x \ge 0$, $\log (1 + \alpha h)\ge \alpha h\,(1-\alpha h / 2)\ge \alpha h/2$ because $\alpha h < 1$. Then it suffices to choose
    \[
    K = \cO \Bigl( \frac{1}{\alpha h} \log \frac{d\log(2\kappa)}{\eps^2} \Bigr)\,.
    \]
    Moreover, by the bounds on total variation in \S~\ref{subsection:analysis-along-RGO-hat}, we have
        $\|\rho_K - \hat \rho_K\|_\TV \le 3K\zeta$, where each iteration has complexity $N = \cO(\log\frac{1}{\zeta})$.
    If both terms are bounded by $\eps/2$, then the proof is complete. Furthermore, if the latter term is bounded by $\eps/2$, so is the former because $K \in \bN^\ast$. Hence, we can let $\zeta = \eps/(6K)$. By plugging in $h$, $K$ depends on $\zeta$ and itself as following:
    \[
    K = \cO \Bigl( \kappa \sqrt{d\log(2\kappa)} \log^2 \frac{1}{\zeta} \log \frac{\sqrt{d \log (2\kappa)}}{6K\zeta} \Bigr) \,.
    \]
    Neglecting logarithmic terms on $d$ and $\kappa$, we obtain that $K = \widetilde \cO (\kappa \sqrt{d} \log^3 \frac{1}{\zeta})$. Since $N = \cO(\log\frac{1}{\zeta})$ and $\zeta = \eps/(6K)$, the final complexity is 
    \begin{align*}
    KN &= \widetilde \cO \Bigl( \kappa \sqrt{d} \log^4 \frac{1}{\eps} \Bigr)\,. 
    \end{align*}
\end{proof}

\begin{proof}[Proof of Theorem~\ref{thm:extensions}]
\label{app:proof-of-extensions}
We use the proximal sampler algorithm.
In all cases, the number of outer loop iteration $K$ is given by Lemma~\ref{lem:KL-convergence-outer-loop} to be
\begin{align}\label{eq:outer_loop_bds}
    K = \cO\Bigl(\frac{1}{\alpha h} \log \frac{\KL(\rho_0\|\pi)}{\eps^2}\Bigr)\,, \quad K = \cO\Bigl(\frac{1}{\alpha h} \log \frac{\chi^2(\rho_0\|\pi)}{\eps^2}\Bigr)\,, \quad K =\cO\Bigl(\frac{W_2^2(\rho_0,\pi)}{h\eps^2}\Bigr)
\end{align}
under $\LSI(1/\alpha)$, $\PI(1/\alpha)$, and log-concavity, respectively. Then, we only need to consider the complexity $N$ of the implementation of the RGO in each case, where we use $\zeta \in (0,1)$ to denote the error of the implemented RGO to the true RGO.

\paragraph{Smooth case} Here, $f$ is $\beta$-smooth. Choose $h = 1/(2\beta)$ and for $y \in \bR^d$, define 
\[
f_h (x) \deq  f(x) + \frac{1}{2h}\, \norm{x-y}^2 \,.
\]
Note that $\nabla^2 f_h = \nabla^2 f + \frac{1}{h} I_d$, so $f_h$ is $\beta$-strongly convex and $3\beta$-smooth. Thus, the RGO $\pi^{X \mid Y}$, whose density is proportional to 
\[
\exp\Bigl( -f(x) - g(x) - \frac{1}{2h}\, \norm{x-y}^2 \Bigr) = \exp\bigl(-f_h(x) - g(x) \bigr) \,,
\]
is a composite log-concave density with condition number $\kappa \le 3$. By Theorem~\ref{thm:main-result}, we can implement $\pi^{X \mid Y_k}$ in $N = \widetilde \cO(\sqrt{d} \log^4 \frac{1}{\zeta})$ function and gradient evaluations at each $k$ given access to the mode of $\pi^{X \mid Y_k}$, i.e., $x_{k, \ast} \deq \prox_{h(f+g)}(Y_k)$. Combining $\zeta = \eps / (6K)$ with $K$ defined in Eq~\eqref{eq:outer_loop_bds}, we obtain the desired complexities under the three different conditions on $\pi$.

\paragraph{Non-smooth case}
Here, $f$ is $L$-Lipschitz.
In this case, according to Corollary~\ref{cor:RGOhat-is-close}, we can choose $h\asymp 1/\bigl(L^2 \log^3(1/\zeta)\bigr)$, so $N = \cO(\log(1/\zeta))$. Combining this with the corresponding $K$ defined in Eq~\eqref{eq:outer_loop_bds} yields the results.
\end{proof}

\subsection{Reduction to shared minimizer}
\label{ssec:shared-reduction}
As in the remark of Theorem \ref{thm:main-result}, in the case where $f$ is $\alpha_f$-strongly convex and $\beta$-smooth, with $-\beta < \alpha_f < 0$, we can use a similar approach as in \cite{LeeSheTia21RGO} to perform reduction to the shared minimizer.
\begin{lemma}
    \label{lem:shared-minimizer}
    Let $f: \bR^d \mapsto \bR$ be $\alpha_f$-strongly convex and $\beta$-smooth and $g: \bR^d \mapsto \bR \cup \{\infty\}$ be proper, closed, and $\alpha_g$ convex, where $0 \le \alpha_g \le \beta$ and $\alpha \deq \alpha_f + \alpha_g > 0$, which permits $\alpha_f \in (-\beta, 0)$. Denote $x_\ast = \arg\min \{f + g\}$. Define 
    \begin{align*}
        \tilde f(x) &\deq f(x) - \langle \nabla f(x_\ast), x - x_\ast \rangle + \frac{\alpha_g}{2} \|x - x_\ast\|^2 \,, \\
        \tilde g(x) &\deq g(x) + \langle \nabla f(x_\ast), x - x_\ast \rangle - \frac{\alpha_g}{2} \|x - x_\ast \|^2 \,.
    \end{align*} 
    Then $\tilde f + \tilde g = f + g$ and $x_\ast$ minimizes both $\tilde f$ and $\tilde g$. Moreover, $\tilde f$ is $\alpha$-strongly convex and $\tilde \beta$-smooth with $\tilde \beta \deq \beta + \alpha_g \le 2 \beta$, and $\tilde g$ is convex.
\end{lemma}
\begin{proof}
    By definition, $\tilde f + \tilde g = f + g$ always holds, so do the convexity and smoothness claims. It suffices to show that $x_\ast$ minimizes both $\tilde f$ and $\tilde g$. Thus, by first-order optimality, $-\nabla f(x_\ast) \in \partial g(x_\ast)$. Let $x \in \bR^d$. 
    \begin{align*}
        \tilde f(x) - \tilde f(x_\ast) &= f(x) - \langle \nabla f(x_\ast), x- x_\ast \rangle + \frac{\alpha_g}{2} \|x - x_\ast\|^2 - f(x_\ast) \\
        &\ge \frac{\alpha_g+\alpha_f}{2} \| x - x_\ast \|^2 = \frac{\alpha}{2} \|x - x_\ast \|^2 \ge 0 \quad \text{by $\alpha_f$-convexity of $f$.} \\
        \tilde g(x) - \tilde g(x_\ast) &= g(x) + \langle \nabla f(x_\ast), x - x_\ast \rangle - \frac{\alpha_g}{2} \|x - x_\ast \|^2 - g(x_\ast) \\
        &\ge 0 \quad \text{by $\alpha_g$-convexity of $g$ and $-\nabla f(x_\ast) \in \partial g(x_\ast)$.}
    \end{align*}
\end{proof}
The reduction preserves $\alpha$ and increases the smoothness
parameter to $\tilde\beta = \beta + \alpha_g \le 2\beta$, so the new condition number $\tilde \kappa \le 2\kappa$, which affects our stated complexity bounds only through universal constant factors. By direct substitution assuming $h \le 1/(2\beta)$, the RGO of $\tilde g$ is then 
\[
    \RGO_{\tilde g, h, y}(x) = \RGO_{g, h', y'}(x) \,, \quad h' = \frac{h}{1-h\alpha_g}\,, \quad y'=\frac{y-h(\nabla f(x_\ast) + \alpha_g x_\ast)}{1 - h\alpha_g} \,.
\]

\section{Proof of independent Metropolis--Hastings}

In this section, we study the \textbf{independent} Metropolis--Hastings algorithm.
For any $x \in \bR^d$, let the proposal kernel be $Q(x, \cdot) \deq  Q_x \deq \mu$.
The transition kernel of the Metropolis-adjusted chain is defined to be
\[
P(x, \dd y) \deq Q(x, \dd y)\, A(x,y) + \Bigl( 1 - \int Q(x,\dd y')\, A(x,y') \Bigr)\, \delta_x(\dd y)\,,
\]
where $\delta_x$ is the Dirac mass at $x$ and 
\[
A(x,y) = \min \Bigl(1, \frac{\pi(y)\, Q(y,x)}{\pi(x)\, Q(x,y)} \Bigr) = \min \Bigl( 1, \frac{\pi(y)\,\mu(x)}{\pi(x)\,\mu(y)} \Bigr)
\]
is the acceptance probability.
Here, $\pi$ is the stationary distribution.
\begin{remark}
    \label{rm:lazy-chains}
    We establish the conductance bound for the non-lazy Markov chain $P$; its $\frac{1}{2}$-lazy version $\tilde P = (I + P)/2$ has half the conductance, so the same mixing bound holds up to an absolute constant \cite{lovasz1993random, vempala2005geometric, mou2022efficient}.  
\end{remark}

\subsection{Preliminaries on conductance analysis}

We begin by recalling the standard toolbox for conductance analysis. Suppose we have a Markov chain on $\bR^d$ with $P$ as its transition kernel and $\pi$ be the stationary distribution. 
\begin{defin}[Cheeger isoperimetry]
We say that $\pi$ satisfies a Cheeger isoperimetric inequality with constant $\mathsf{Ch} > 0$ if
\begin{align*}
    \liminf_{\eps\to 0^+} \frac{\pi(A^\eps) - \pi(A)}{\eps} \ge \frac{1}{\mathsf{Ch}}\,\pi(A)\,(1-\pi(A))
\end{align*}
for all measurable sets $A \subseteq \bR^d$, where $A^\eps \deq \{x\in\bR^d : \dist(x, A) < \eps\}$ is the $\eps$ blow-up of $A$.
\end{defin}

If $\pi$ is $\alpha$-strongly log-concave, then $\mathsf{Ch} \lesssim 1/\sqrt{\alpha}$ \cite{milman2009role}.

\begin{defin}[$s$-conductance]
    For $s \in [0,\frac{1}{2}]$, the \textit{$s$-conductance} of $P$ is the largest $\fc_s > 0$ such that for all events $A \subseteq \bR^d$, 
    \[
    \int_A P(x, A^c)\, \pi(\dd x) \ge \fc_s\, (\pi(A) - s)\, (\pi(A^\comp) - s) \,. 
    \]
\end{defin}
Note that for $\pi(A) \le s$, the inequality holds trivially. 

\begin{lemma}[\cite{lovasz1993random}, Corollary 1.6]
    \label{lem:s-conductance}
    For any $s \in (0,\frac{1}{2}]$, let 
    \[
    \Delta_s \deq \sup \{ |\mu_0(A) - \pi(A)| : A \subseteq \bR^d,\, \pi(A) \le s \}\,.
    \]
    Then, the law of $\mu_N$ of the $N$-th iterate of a Markov chain with $s$-conductance $\fc_s$, stationary distribution $\pi$, and initialized at $\mu_0$ satisfies 
    \[
    \|\mu_N - \pi\|_{\TV} \le \Delta_s + \frac{\Delta_s}{s}\exp \Bigl( - \frac{\fc_s^2 N}{2} \Bigr)\,.
    \]
    In particular, 
    \[
    \|\mu_N - \pi\|_{\TV} \le \sqrt{s \chi^2(\mu_0 \| \pi)} + \sqrt{\frac{\chi^2(\mu_0 \| \pi)}{s}} \exp \Bigl( - \frac{\fc_s^2 N}{2} \Bigr)\,.
    \]
\end{lemma}
\begin{remark}
    \label{rm:subroutine-complexity}
    Lemma \ref{lem:s-conductance} implies that, by choosing $s = \frac{\eps^2}{4\chi^2(\mu_0\|\pi)}$, we have $\|\mu_N - \pi\|_\TV \le \eps$ after 
    \[
    N = \cO\Bigl( \frac{1}{\fc_s^2} \log \frac{\chi^2(\mu_0\|\pi)}{\eps^2}  \Bigr) \quad \text{iterations.}
    \]
\end{remark}

To lower bound the $s$-conductance, we use the following geometric overlap lemma.
\begin{lemma}[Overlap lemma for $s$-conductance, \cite{chewi2026log}, Ch.~7]
    \label{lem:overlap-s-conductance}
    Assume:
    \begin{enumerate}
        \item $\pi$ satisfies Cheeger isoperimetric inequality with constant $\mathsf{Ch} > 0$.
        \item There exists $r \in [0, \mathsf{Ch}]$ and $E \subseteq \bR^d$ with probability $\pi(E) \geq 1 - \frac{rs}{16 \mathsf{Ch}}$ such that 
        \[
        \forall x, y \in E\,, \; \|x-y\| \leq r \Longrightarrow \|P(x, \cdot) - P(y, \cdot)\|_{\TV} \le \frac{1}{2}\,.
        \]
    \end{enumerate}
    Then $\fc_s \gtrsim r/\mathsf{Ch}$.
\end{lemma}

Denote $P(x, \cdot)$ by $P_x$. In order to apply Lemma~\ref{lem:overlap-s-conductance}, we need to control $\|P_x - P_y\|_\TV \le \frac{1}{2}$ for $x, y \in \bR^d$ on a set $E$ with high probability under $\pi$. By the triangle inequality, 
\begin{align}
    \| P_x - P_y \|_\TV
    &\le \|P_x - Q_x\|_\TV + \|Q_x - Q_y\|_\TV + \|Q_y - P_y\|_\TV \nonumber \\
    &= \|P_x - Q_x \|_\TV + \|P_y - Q_y\|_\TV\,,\label{eq:mh_triangle}
\end{align}
where the middle term is $0$ for the independent MH algorithm.
To control the other terms, we invoke the following lemma.

\begin{lemma}[Pointwise projection property, \cite{chewi2021optimal}, Theorem 6]
    \label{lem:pointwise-projection-property}
    Let $Q$ be an atomless proposal kernel and $P$ be the corresponding Metropolis--Hastings kernel with target distribution $\pi$. Then, for any atomless kernel $\Bar{Q}$ which is reversible with respect to $\pi$, and for every $x \in \bR^d$, 
    \[
    \|P_x - Q_x\|_{\TV} \leq 2\, \|Q_x - \Bar{Q}_x\|_{\TV} + \int \frac{\pi(y)\, \Bar{Q}(y,x)}{\pi(x)} \,\Bigl| \frac{Q(y,x)}{\Bar{Q}(y,x)} -1 \Bigr|\, \dd y\,.
    \]
\end{lemma}

We next show how to control these terms under a R\'enyi divergence condition for the proposal.

\subsection{Proof of Theorem~\ref{thm:indep_mh}}
\label{app:proof-of-indep-mh}
\begin{proof}[Proof of Theorem~\ref{thm:indep_mh}]
    Since both $\mu$ and $\pi$ are atomless, we can apply Lemma~\ref{lem:pointwise-projection-property} and obtain that: $\forall x \in \bR^d$,
    \begin{equation}
        \label{eq:pointwise-projection-IMHA}
        \| P_x - Q_x\|_\TV \le 2\, \| \mu - \pi\|_\TV + \int \frac{\pi(y)\, \pi(x)}{\pi(x)} \,\Bigl| \frac{\mu(x)}{\pi(x)}-1 \Bigr|\, \dd y = 2\, \|\mu - \pi\|_\TV + \Bigl| \frac{\mu(x)}{\pi(x)}-1 \Bigr| \,.
    \end{equation}
    For the first term, by Pinsker's inequality,
    \begin{equation}
        \label{eq:TV-bound-IMHA}
        \|\mu - \pi \|_\TV \le \sqrt{\frac{1}{2}\, \KL(\mu \| \pi)} = \sqrt{\frac{1}{2}\, \sfR_1(\mu \| \pi)} \le \sqrt{\frac{1}{2}\, C_\sfR} \, .
    \end{equation}
    For the second term, by Markov's inequality: $\forall t \in (0,1)$, 
    \begin{align}
        \label{eq:ratio-bound-IMHA}
        \pi \Bigl( \left| \frac{\mu}{\pi}-1 \right| \ge t\Bigr) &= \pi \Bigl( \frac{\mu}{\pi} \ge 1+t \Bigr) + \pi \Bigl( \frac{\pi}{\mu} \ge \frac{1}{1-t} \Bigr) \notag \\
        &\le (1+t)^{-q}\, \bE_\pi\left[(\mu/\pi)^q \right] + (1-t)^{q-1}\, \bE_\pi[(\pi/\mu)^{q-1}] \notag \\
        &=(1+t)^{-q} \exp((q-1)\, \sfR_q(\mu \| \pi)) + (1-t)^{q-1} \exp( (q-1)\, \sfR_q(\pi \| \mu)) \notag \\
        &\le ((1+t)^{-q}+(1-t)^{q-1}) \exp(C_\sfR q^{\gamma +1})\,.
    \end{align}
    With $t = \frac{1}{8}$, Eq.~\eqref{eq:ratio-bound-IMHA} becomes: 
    \[
    \pi\Bigl( \left| \frac{\mu}{\pi}-1\right| \ge \frac{1}{8} \Bigr) \le \Bigl[ \bigl( \frac{8}{9} \bigr)^q + \bigl( \frac{7}{8}\bigr)^{q-1}\Bigr] \exp(C_\sfR q^{\gamma +1}) \le 3 \,\bigl(\frac{8}{9}\bigr)^q \exp(C_\sfR q^{\gamma+1}) \, .
    \]
    Define the set $E  \deq  \{ |\mu/\pi -1| \leq \frac{1}{8}\}$. 
    By the triangle inequality in Eq.~\eqref{eq:mh_triangle}, Eq.~\eqref{eq:pointwise-projection-IMHA}, Eq.~\eqref{eq:TV-bound-IMHA}, and Eq.~\eqref{eq:ratio-bound-IMHA}, we have
    \begin{align*}
       \|P_x - P_y\|_\TV
       &\le 4\, \| \mu - \pi \|_\TV + \Bigl|\frac{\mu(x)}{\pi(x)}-1\Bigr| + \Bigl|\frac{\mu(y)}{\pi(y)}-1\Bigr| \le \frac{1}{2} 
    \end{align*}
    for all $x,y\in E$, provided that $C_\sfR \le \frac{1}{128}$.
    
    As in Remark~\ref{rm:subroutine-complexity}, let $s = \frac{\eps^2}{4 \chi^2(\mu \| \pi)}$, where we initialize the Markov chain with $\mu_0 = \mu$. Then,
    \[
    \chi^2 (\mu \| \pi) = \exp(\sfR_2(\mu\|\pi)) -1 \le \exp (C_\sfR\,2^{\gamma}) = \cO(1)\, ,
    \]
    so $s \asymp \eps^2$. In Lemma~\ref{lem:overlap-s-conductance}, we can choose $r = \mathsf{Ch}$. To bound $\pi(E^\comp)$ by $\frac{rs}{16 \mathsf{Ch}} = \frac{s}{16}$, it suffices to take $q \asymp \log(1/s) \asymp \log(1/\eps)$ and for $C_\sfR \ll q^{-(\gamma+1)}$.
    Then, by Lemma~\ref{lem:overlap-s-conductance}, the $s$-conductance satisfies $\fc_s \gtrsim r/\mathsf{Ch} =1$. Finally, according to Remark~\ref{rm:subroutine-complexity}, the number of iterations we need to obtain $\eps$-TV distance to $\pi$ is 
    \begin{align*}
    N &= \cO\Bigl( \frac{1}{\fc_s^2} \log \frac{\chi^2(\mu \| \pi)}{\eps^2} \Bigr) = \cO(\log(1/\eps)) \,. 
    \end{align*}
\end{proof}

\section{RGO implementations}

We give examples of functions $g$ for which the RGO for $g$ is easily implementable. Given a set $\eu C$, we let $\iota_{\eu C}$ denote the convex indicator of $\eu C$. We denote $\Phi$ as the standard Gaussian CDF.
    \begin{itemize}
        \item If $g_1,\dotsc,g_k$ admit easily implementable RGOs and $g(x_1,\dotsc,x_k)\deq \sum_{i=1}^k g_i(x_i)$, then $\RGO_{g, h, y} = \bigotimes_{i=1}^k \RGO_{g_i, h, y_i}$.
        In particular, since RGOs for one-dimensional functions are generally easy to implement (e.g., via rejection sampling), separable functions admit implementable RGOs. We cover numerous examples such as indicators of boxes, the $\ell_1$ penalty $g(x) \deq \lambda\,\norm x_1$, the $\ell_\infty$ penalty $g(x) \deq \lambda\,\norm x_\infty$, and others. 
        \begin{itemize}
            \item \textbf{Indicator of a box.}
            For the indicator of a box, $g = \iota_{\bigtimes_{i=1}^d [a_i ,b_i]}$, sampling from the RGO reduces to sampling from one-dimensional truncated Gaussians $\cN(y, h)|_{[a,b]}$.
            \item \textbf{$\ell_1$ penalty.}
            For $g(x) \deq \lambda\,\norm x_1$, by separability, it suffices to consider the one-dimensional case.
            One can see that
            \begin{align*}
                \RGO_{g,h,y}(x) \propto \begin{cases}
                    \exp(-\frac{1}{2h}\,(x-(y+\lambda h))^2)\,,& x \le 0\,, \\
                    \exp(-\frac{1}{2h}\,(x-(y-\lambda h))^2)\,, & x \ge 0\,.
                \end{cases}
            \end{align*}
            Define the weights
            \[
            w_- \deq \exp\bigl(\lambda y + \frac{1}{2}\,\lambda^2 h\bigr)\,\Phi\bigl(-\frac{y+\lambda h}{\sqrt h}\bigr)\,, \qquad w_+ \deq \exp\bigl(-\lambda y + \frac{1}{2}\,\lambda^2 h\bigr)\,\Phi\bigl(\frac{y-\lambda h}{\sqrt h}\bigr)\,,
            \]
            Then, an exact sample from $\RGO(g,h,y)$ is produced as follows: with probability $w_-/(w_- + w_+)$, draw a sample from $\cN(y+\lambda h, h)|_{\bR_-}$; otherwise, draw a sample from $\cN(y-\lambda h, h)|_{\bR_+}$.
            \item \textbf{$\ell_\infty$ penalty.}
            For $g(x) \deq \lambda\,\norm x_\infty$, the density of the RGO is
            \begin{align*}
                \RGO_{g,h,y}(x)
                &\propto \exp\Bigl(-\lambda\,\norm x_\infty - \frac{1}{2h} \sum_{i=1}^d (x_i - y_i)^2\Bigr) \\
                &\propto \int_0^\infty \exp\Bigl(-\lambda t - \frac{1}{2h} \sum_{i=1}^d (x_i - y_i)^2\Bigr)\one_{t\ge \norm x_\infty}\,\dd t \\
                &\propto \int_0^\infty \exp(-\lambda t)\,\prod_{i=1}^d \Bigl[Z_i(t)\,\frac{\exp(-\frac{1}{2h}\,(x_i-y_i)^2) \one_{t\ge \abs{x_i}}}{Z_i(t)}\Bigr]\,\dd t\,,
            \end{align*}
            where $Z_i(t) = \sqrt{2\pi h}\,[\Phi(\frac{t-y_i}{\sqrt h}) - \Phi(\frac{-t-y_i}{\sqrt h})]$.
            This admits the following interpretation: first, sample $t$ from the following density over $\bR_+$.
            \begin{align*}
                p(t) \propto \exp(-\lambda t) \prod_{i=1}^d Z_i(t)\,.
            \end{align*}
            Conditioned on $t$, draw $X_1,\ldots,X_d$ independently with $X_i \sim \cN(y_i, h)|_{[-t, t]}$. The max penalty $g(x) = \lambda\max_{i\in [d]} x_i$ can be handled similarly.
            \item \textbf{Other separable functions.}
            Further examples of separable functions include $\ell_p$ norms $g(x) \deq \lambda\,\norm x_p^p$ for $p\ge 1$; the elastic net $g(x) \deq \lambda_1\,\norm x_1 + \frac{\lambda_2}{2}\,\norm x_2^2$; hinge penalties $g(x) \deq \sum_{i=1}^d \lambda_i\,(x_i - b_i)_+$; and the entropy map $g(x) \deq \sum_{i=1}^d x_i \log x_i + \iota_{\bR_+^d}(x)$.
            In each of these cases, the RGO reduces to sampling from one-dimensional densities, for which closed-form samplers or tailored rejection sampling schemes can be developed.
        \end{itemize}
    \item If $g(x) \deq \tilde g(x-c)$ for some $c\in\bR^d$, then $\RGO_{g,h,y} = (x\mapsto x+c)_\sharp\RGO_{\tilde g, h, y-c}$.
    \item If $g(x) \deq \tilde g(x) + \frac{a}{2}\,\norm x^2 + \langle b, x\rangle + c$, then $\RGO_{g, h, y} = \RGO_{\tilde g, \frac{h}{1+ah}, \frac{y-hb}{1+ah}}$.
    \item If $g(x) \deq \bar g(Bx)$ for a matrix $B \in \bR^{k\times d}$ of full rank, then a sample from $\RGO_{g, h, y}$ is produced as follows. First, sample $\bar x \in \bR^k$ from the density
    \begin{align*}
        \mu(\bar x) \propto \exp\bigl(-\bar g(\bar x)-\frac{1}{2h}\,\langle \bar x-By, (BB^\T)^{-1}\,(\bar x-By)\rangle\bigr)\,.
    \end{align*}
    Then, sample $x \sim \cN(y+B^\T (BB^\T)^{-1} (\bar x-By),\, h (I_d-B^\T (BB^\T)^{-1} B))$.

    In particular, when $k=1$ and $B = b^\T$, $\bar\mu = \RGO_{\bar g, h\norm b^2, \langle b,y\rangle}$ is a one-dimensional RGO. This covers examples such as convex indicators of half-spaces $\{\langle b,\cdot\rangle \le c\}$ and slabs $\{\underline c \le \langle b,\cdot\rangle \le \overline c\}$.
    \item For a quadratic $g(x) \deq \frac{1}{2}\,\langle x,A\,x\rangle + \langle b,x\rangle + c$, $\RGO_{g, h, y} = \cN(A_h(y/h-b), A_h)$, where $A_h^{-1} = A + h^{-1} I_d$, provided that $A + h^{-1} I_d \succ 0$.
\end{itemize}

\section{Numerical experiments}

We test the effectiveness of Algorithm~\ref{alg:outer-loop-Gibbs-Sampling} of the main paper (Composite Sampler) with two baseline algorithms~\cite{pereyra2016proximal}, namely, the proximal Metropolis-adjusted Langevin algorithm (Prox--MALA), which adds a Metropolis--Hastings correction to the proposal
\begin{align*}
    Q(x,\cdot) = \cN(\prox_{hg}(x-h\,\nabla f(x)), 2hI_d)\,,
\end{align*}
and the proximal gradient Langevin algorithm (PGLA), whose iterates are 
\begin{align*}
    x_{k+1} = \prox_{hg}(x_k - h\,\nabla f(x_k) + \cN(0, 2hI_d))\,.
\end{align*}
We tune the step sizes for each method to optimize performance.
For the first two experiments, we take $f$ to correspond to the negative log-likelihood function for logistic regression:
\begin{align*}
    f(x) = \sum_{i=1}^n \bigl\{\log(1+\exp(\langle a_i,x\rangle)) - y_i\,\langle a_i,x\rangle\bigr\} + \frac{\tau}{2}\,\norm x_2^2\,,
\end{align*}
where $\tau > 0$, and $a_i \in \bR^d$ are the rows of the design matrix $A \in \bR^{n \times d}$.

As this is primarily a theoretical paper, the following three experiments are not meant to be comprehensive and only serve to illustrate a proof of concept. The code chunks run in less than a minute or two in total by the default runtime type on Google Colab for the first two, and it takes around $10$ minutes to run the third.

For each method, we roughly counted the number of gradient evaluations $\nabla f$ per iteration as follows. (The number of function evaluations of $f$ is roughly the same.)
\begin{itemize}
    \item PGLA always uses one gradient evaluation per iteration.
    \item Prox-MALA uses two gradient evaluations per iteration: one is used to generate the proposal, and another is needed to compute the acceptance ratio.
    \item In each outer iteration of the Composite Sampler, we use one initial computation of $\nabla f$ for the first proposal.
    Each inner iteration for the RGO implementation further requires one new gradient evaluation in order to evaluate the acceptance probability.
    In total, each outer iteration requires one evaluation of $\prox_{hf}$ and $N+1$ evaluations of $\nabla f$, where $N$ is the number of inner iterations.
    In practice, we capped the number of inner iterations to be $N_{\max}$ and we count the cost of an outer iteration to be $N_{\max}+1$, although this somewhat overcounts since the number of actual gradient evaluations can be smaller if the proposals are accepted in fewer than $N_{\max}$ steps.
\end{itemize}

\paragraph{Code availability.}
The code, numerical results, and figures for the experiments
are available at
\url{https://github.com/lliu58b/composite_sampler}.

\subsection{Sparse Bayesian logistic regression with \texorpdfstring{$\ell_1$}{l1} prior}

Here, we take the $\ell_1$ prior $g = \lambda\,\norm \cdot_1$. We use the following settings: $d=36$, $n=360$, $\tau=0.2$, $\lambda =7$; the design matrix $A$ has rows drawn from $\cN(0,\Sigma)$ rescaled to have norm $\sqrt n$, where $\Sigma_{i,j} = \rho^{\abs{i-j}}$ and $\rho = 0.65$; the ground truth is chosen to be sparse: $x_{\rm true} \deq (1, -1, 0.8, 1.2, -0.9,0,\dotsc,0)$; and the labels $\{y_i\}_{i\in [n]}$ are drawn from the logistic regression model with ground truth $x_{\rm true}$.

In Figure~\ref{fig:logistic_l1}, we plot the running root mean-squared error (RMSE) $\frac{1}{\sqrt{d}} \,\norm{ \frac{1}{k}\sum_{j=1}^k x_j - \widehat x}_2$,  where $\widehat x$ is the posterior mean, as well as 90\% coverage intervals for the coordinates.
The ground truths are computed from a long, conservative run of Prox-MALA.

\begin{figure}[ht]
    \centering
    \includegraphics[width=0.49\linewidth]{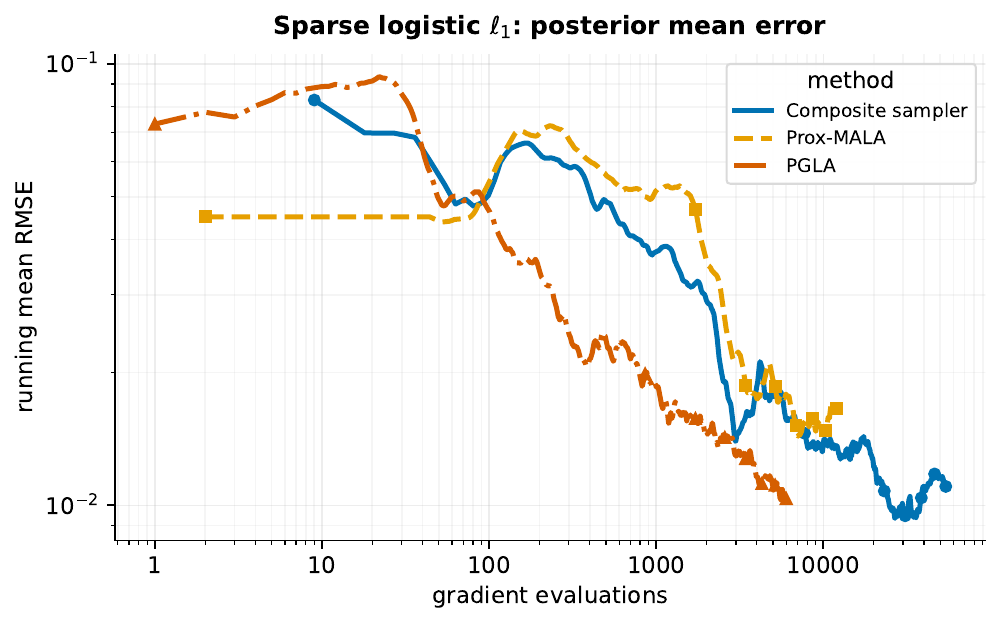}
    \includegraphics[width=0.49\linewidth]{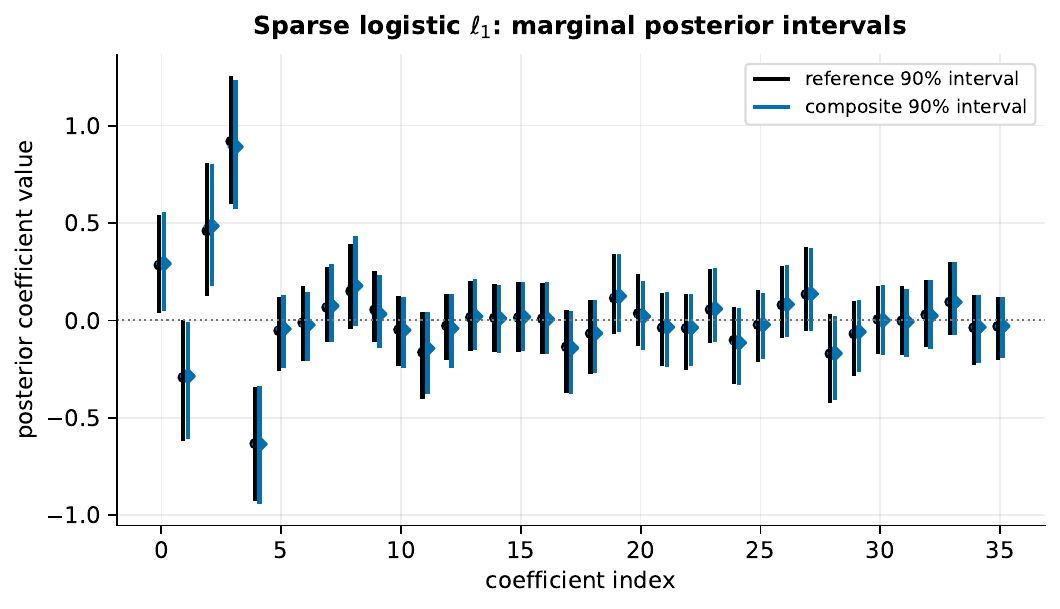}
    \caption{(Left) RMSE comparison between the Composite Sampler, Prox-MALA, and PGLA for sparse Bayesian logistic regression with an $\ell_1$ prior. (Right) Marginal posterior coverage of the Composite Sampler.}
    \label{fig:logistic_l1}
\end{figure}

The experiment reveals that our composite sampler is competitive with Prox-MALA, and achieves good marginal coverage.
In this experiment, PGLA noticeably converges to a smaller RMSE with fewer iterations; however, this is not the end of the story.
Due to the application of $\prox_{hg}$, 11.3\% of the coordinates of PGLA were \emph{exactly} $0$ at the end of the run, which does not accurately reflect the posterior distribution. Indeed, PGLA suffers from asymptotic bias for $h > 0$.

\subsection{Bayesian logistic regression with box constraints}

Here we take the convex indicator of a box, $g = \iota_{[-R,R]^d}$ and use the following settings: $d=24$, $n=360$, $\tau=0.2$, and $R=0.35$. The data is generated similarly as before, except that $\rho = 0.55$ and the ground truth vector is $x_{\rm true} = (3, -3, 3, -3, 0,\dotsc, 0)$.

In this setting, the Composite Sampler clearly outperforms Prox-MALA, as shown in Figure~\ref{fig:logistic_box}; this appears to be because many of the proposals for Prox-MALA are infeasible, which is not a problem for the Composite Sampler since it samples each coordinate from the truncated Gaussian directly. PGLA is competitive in terms of RMSE but again suffers from bias: 10.9\% of the final coordinates lie on the boundary.

\begin{figure}[ht]
    \centering
    \includegraphics[width=0.55\textwidth]{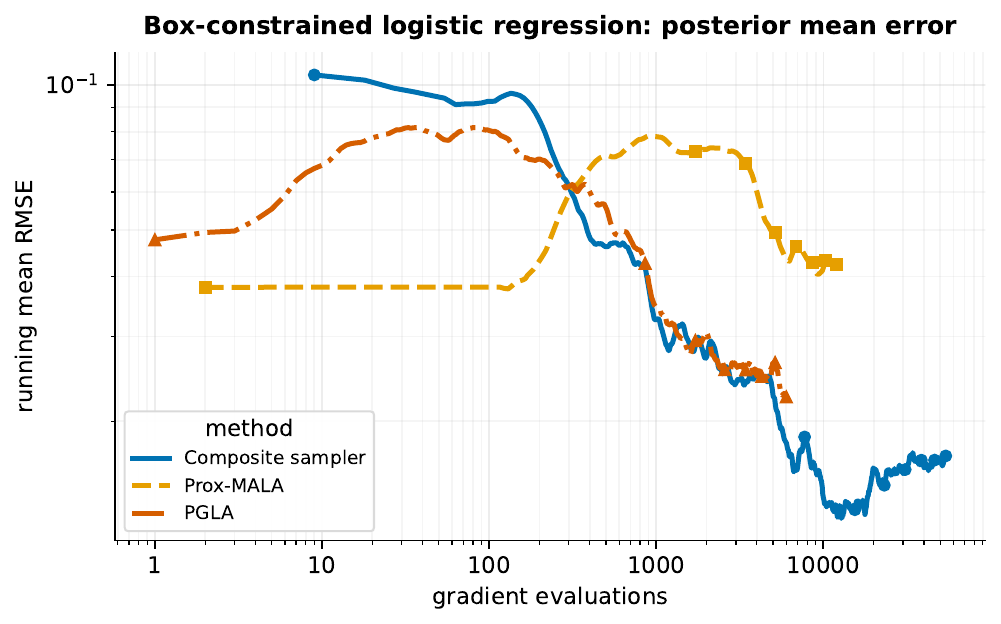}
    \caption{RMSE comparison between the Composite Sampler, Prox-MALA, and PGLA for sparse Bayesian logistic regression with a box constraint.}
    \label{fig:logistic_box}
\end{figure}

\subsection{Gaussian distribution with box constraints}

In this experiment, we take the convex indicator of a box, $g = \iota_{[-R,R]^d}$, and let $\kappa = 1$. Specifically, our target is $\pi(x) \propto \exp(-\frac{1}{2}\norm{x}^2)\mathbbm{1}_{[-R,R]^d}$ with $R = 1$. We then vary $d \in \{4, 8, 16, \ldots, 512\}$, and run each method with $3$ random seeds. Convergence is measured by sliced $W_2$ distance against the target $\pi$. For each method, we set the maximum number of iterations to be $10^6$.

Figure \ref{fig:gaussian-box} reports the results. It is worth noting that the Composite Sampler reaches the threshold at every $d$ tested, and the cost grows roughly as $\sqrt{d}$ — fitting a line through the eight points gives a log-log slope of approximately $0.5$, which agrees with Theorem~\ref{thm:main-result} of the main paper. Prox-MALA reaches the threshold within budget for $d \le 32$; PGLA also only reaches the threshold for some $d$'s, but the dependence of complexity on $d$ is not clear, which may be the variability brought by random seeds.
\begin{figure}[ht]
    \centering
    \includegraphics[width=0.49\linewidth]{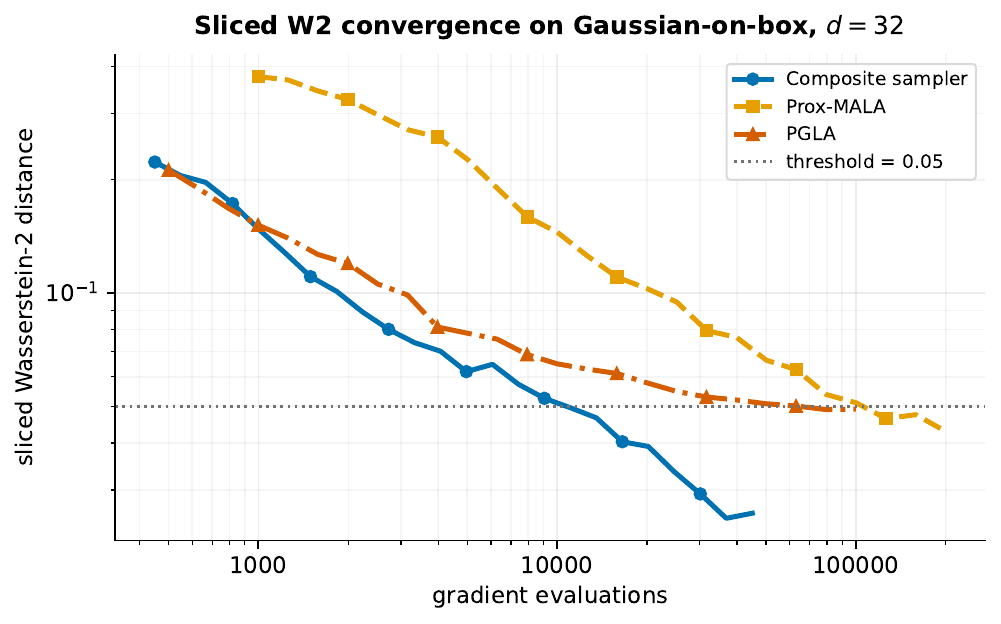}
    \includegraphics[width=0.49\linewidth]{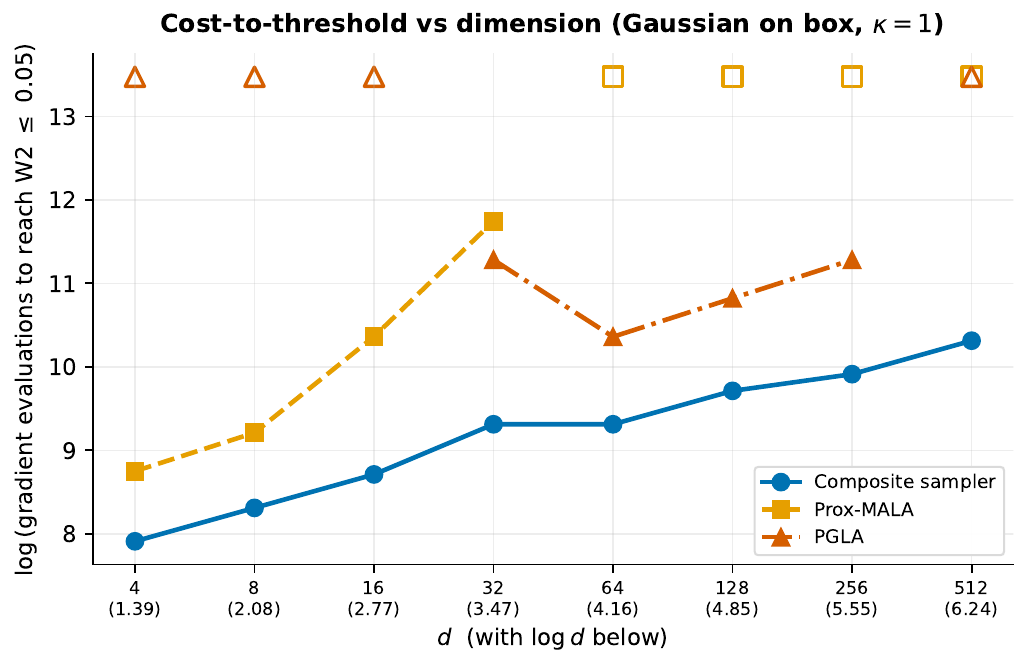}
    \caption{(Left) Sliced $W_2$ distance to the exact target versus gradient evaluations, $d=64$. (Right) Gradient evaluations needed to first reach the threshold for each method, plotted against $d$ on log-log axes. Hollow markers near the top of the right panel denote runs that did not reach the threshold within the per-method gradient budget.}
    \label{fig:gaussian-box}
\end{figure}

\section{Conclusion}

In this paper, we have introduced an algorithm for sampling from composite log-concave target distributions of the form $\pi \propto e^{-f-g}$, inspired by the proximal sampler \cite{LeeSheTia21RGO, chen2022improved}. With access to the restricted Gaussian oracle (RGO) of the potentially non-smooth component $g$, we have shown that our method achieves high-accuracy sampling guarantees with improved dimensional dependence $\widetilde \cO(\sqrt{d})$, matching the current state-of-the-art results for $g \equiv 0$. The key insight is to linearize $f$ in the RGO to control the R\'enyi divergence between the true RGO and our implementation via the tractable proposal ($\RGO_{g,h,y-h\nabla f(y)}$) at step size $h = \widetilde \Theta (1/(\beta \sqrt{d}))$. This larger step size is precisely what drives the improved dimensional dependence. Moreover, we provide verifiable R\'enyi divergence conditions for independent Metropolis--Hastings in our analysis of the implemented RGO. We have also generalized our results to settings in which $\pi$ satisfies weaker conditions, such as isoperimetric inequalities, and to cases in which $f$ is Lipschitz. 

\paragraph{Limitations} 
Several questions still remain open. Our algorithms assume that the non-smooth component $g$ admits an efficient RGO. While we have given many common examples, including $\ell_1$ and $\ell_\infty$ penalties, box and half-space indicators, separable functions, and quadratics, it may still be difficult for more complicated composite structures. Our experimental results, reported in the supplementary material, are preliminary, and a more extensive study such as in high-dimensional Bayesian inverse problems and constrained sampling settings would help to fully understand the performance of our algorithms in practice at scale. A natural future direction could be an adaptive variant of the step size $h$, which would make the implementation easier, since our current choice $h\asymp 1/(\beta \sqrt{d\log(2\kappa)} \log^2(1/\zeta))$ depends on the prescribed choice of the accuracy $\zeta$ in the subroutine.

\section*{Acknowledgments} 
We would like to acknowledge Andre Wibisono for his helpful suggestions and for pointing us to additional references.

We did \textbf{not} use AI to generate ideas or proofs. AI was used in conducting numerical experiments. We checked the correctness of the code and assume responsibility for all content.

    \printbibliography

@InProceedings{LeeSheTia21RGO,
  title = 	 {Structured logconcave sampling with a restricted {G}aussian oracle},
  author =       {Lee, Yin Tat and Shen, Ruoqi and Tian, Kevin},
  booktitle = 	 {Proceedings of Thirty Fourth Conference on Learning Theory},
  pages = 	 {2993--3050},
  year = 	 {2021},
  editor = 	 {Belkin, Mikhail and Kpotufe, Samory},
  volume = 	 {134},
  series = 	 {Proceedings of Machine Learning Research},
  month = 	 {8},
  publisher =    {PMLR},
}

@inproceedings{chewi2021optimal,
  title={Optimal dimension dependence of the {M}etropolis-adjusted {L}angevin algorithm},
  author={Chewi, Sinho and Lu, Chen and Ahn, Kwangjun and Cheng, Xiang and Le Gouic, Thibaut and Rigollet, Philippe},
  booktitle={Conference on Learning Theory},
  pages={1260--1300},
  year={2021},
  organization={PMLR}
}

@incollection{ledoux2006concentration,
  title={Concentration of measure and logarithmic {S}obolev inequalities},
  author={Ledoux, Michel},
  booktitle={Seminaire de probabilites XXXIII},
  pages={120--216},
  year={2006},
  publisher={Springer}
}

@book{chewi2026log,
  author={Sinho Chewi},
  title={Log-concave sampling},
  year={2026+},
  note={Book draft. Available at \url{https://chewisinho.github.io/main.pdf}}
}

@article{wilmer2009markov,
  title={Markov chains and mixing times},
  author={Wilmer, Elizabeth L. and Levin, David A. and Peres, Yuval},
  journal={American Mathematical Soc., Providence},
  volume={107},
  year={2009}
}

@article{chen2026high,
  title={High-accuracy sampling for diffusion models and log-concave distributions},
  author={Chen, Fan and Chewi, Sinho and Daskalakis, Constantinos and Rakhlin, Alexander},
  journal={arXiv preprint arXiv:2602.01338},
  year={2026}
}

@article{vempala2019rapid,
  title={Rapid convergence of the unadjusted {L}angevin algorithm: isoperimetry suffices},
  author={Vempala, Santosh and Wibisono, Andre},
  journal={Advances in Neural Information Processing Systems},
  volume={32},
  year={2019}
}

@inproceedings{fan2023improved,
  title={Improved dimension dependence of a proximal algorithm for sampling},
  author={Fan, Jiaojiao and Yuan, Bo and Chen, Yongxin},
  booktitle={The Thirty Sixth Annual Conference on Learning Theory},
  pages={1473--1521},
  year={2023},
  organization={PMLR}
}

@article{dalalyan2017theoretical,
  title={Theoretical guarantees for approximate sampling from smooth and log-concave densities},
  author={Dalalyan, Arnak S.},
  journal={Journal of the Royal Statistical Society Series B: Statistical Methodology},
  volume={79},
  number={3},
  pages={651--676},
  year={2017},
  publisher={Oxford University Press}
}

@article{durmus2017nonasymptotic,
    author = {Alain Durmus and {\'E}ric Moulines},
    title = {Nonasymptotic convergence analysis for the unadjusted {L}angevin algorithm},
    volume = {27},
    journal = {The Annals of Applied Probability},
    number = {3},
    publisher = {Institute of Mathematical Statistics},
    pages = {1551 -- 1587},
    year = {2017},
}

@inproceedings{chen2022improved,
  title={Improved analysis for a proximal algorithm for sampling},
  author={Chen, Yongxin and Chewi, Sinho and Salim, Adil and Wibisono, Andre},
  booktitle={Conference on Learning Theory},
  pages={2984--3014},
  year={2022},
  organization={PMLR}
}

@article{kook2024and,
  title={In-and-out: algorithmic diffusion for sampling convex bodies},
  author={Kook, Yunbum and Vempala, Santosh S and Zhang, Matthew S},
  journal={Advances in Neural Information Processing Systems},
  volume={37},
  pages={108354--108388},
  year={2024}
}

@article{kook2025faster,
  title={Faster logconcave sampling from a cold start in high dimension},
  author={Kook, Yunbum and Vempala, Santosh S.},
  journal={arXiv preprint arXiv:2505.01937},
  year={2025}
}

@inproceedings{kook2025renyi,
  title={R{\'e}nyi-infinity constrained sampling with {$d^3$} membership queries},
  author={Kook, Yunbum and Zhang, Matthew S.},
  booktitle={Proceedings of the 2025 Annual ACM-SIAM Symposium on Discrete Algorithms (SODA)},
  pages={5278--5306},
  year={2025},
  organization={SIAM}
}

@article{LiaChe23Prox,
title={A proximal algorithm for sampling},
author={Jiaming Liang and Yongxin Chen},
journal={Transactions on Machine Learning Research},
year={2023},
note={}
}

@article{wu2022minimax,
  title={Minimax mixing time of the {M}etropolis-adjusted {L}angevin algorithm for log-concave sampling},
  author={Wu, Keru and Schmidler, Scott and Chen, Yuansi},
  journal={Journal of Machine Learning Research},
  volume={23},
  number={270},
  pages={1--63},
  year={2022}
}

@article{altschuler2024faster,
  title={Faster high-accuracy log-concave sampling via algorithmic warm starts},
  author={Altschuler, Jason M and Chewi, Sinho},
  journal={Journal of the ACM},
  volume={71},
  number={3},
  pages={1--55},
  year={2024},
  publisher={ACM New York, NY}
}

@article{chewi2025analysis,
  title={Analysis of {L}angevin {M}onte {C}arlo from {P}oincar\'{e} to log-{S}obolev},
  author={Chewi, Sinho and Erdogdu, Murat A and Li, Mufan and Shen, Ruoqi and Zhang, Matthew S.},
  journal={Foundations of Computational Mathematics},
  volume={25},
  number={4},
  pages={1345--1395},
  year={2025},
  publisher={Springer}
}

@article{dwivedi2019log,
  title={Log-concave sampling: {M}etropolis--{H}astings algorithms are fast},
  author={Dwivedi, Raaz and Chen, Yuansi and Wainwright, Martin J. and Yu, Bin},
  journal={Journal of Machine Learning Research},
  volume={20},
  number={183},
  pages={1--42},
  year={2019}
}

@article{mou2022efficient,
  title={An efficient sampling algorithm for non-smooth composite potentials},
  author={Mou, Wenlong and Flammarion, Nicolas and Wainwright, Martin J. and Bartlett, Peter L.},
  journal={Journal of Machine Learning Research},
  volume={23},
  number={233},
  pages={1--50},
  year={2022}
}

@article{dang2025oracle,
  title={Oracle-based Uniform Sampling from Convex Bodies},
  author={Dang, Thanh and Liang, Jiaming},
  journal={arXiv preprint arXiv:2510.02983},
  year={2025}
}

@book{vershynin2018high,
  title={High-dimensional probability: an introduction with applications in data science},
  author={Vershynin, Roman},
  volume={47},
  year={2018},
  publisher={Cambridge University Press}
}

@article{lovasz1993random,
  title={Random walks in a convex body and an improved volume algorithm},
  author={Lov{\'a}sz, L{\'a}szl{\'o} and Simonovits, Mikl{\'o}s},
  journal={Random Structures \& Algorithms},
  volume={4},
  number={4},
  pages={359--412},
  year={1993},
  publisher={Wiley Online Library}
}

@book {BouLugMas13Con,
    AUTHOR = {Boucheron, St\'{e}phane and Lugosi, G\'{a}bor and Massart, Pascal},
     TITLE = {Concentration inequalities},
      NOTE = {A nonasymptotic theory of independence,
              With a foreword by Michel Ledoux},
 PUBLISHER = {Oxford University Press, Oxford},
      YEAR = {2013},
     PAGES = {x+481},
}

@article{parikh2014proximal,
  title={Proximal algorithms},
  author={Parikh, Neal and Boyd, Stephen},
  journal={Foundations and Trends in Optimization},
  volume={1},
  number={3},
  pages={127--239},
  year={2014},
  publisher={Emerald Publishing Limited}
}

@article{beck2009fast,
  title={A fast iterative shrinkage-thresholding algorithm for linear inverse problems},
  author={Beck, Amir and Teboulle, Marc},
  journal={SIAM Journal on Imaging Sciences},
  volume={2},
  number={1},
  pages={183--202},
  year={2009},
  publisher={SIAM}
}

@article{durmus2019high,
author = {Alain Durmus and {\'E}ric Moulines},
title = {High-dimensional {B}ayesian inference via the unadjusted {L}angevin algorithm},
volume = {25},
journal = {Bernoulli},
number = {4A},
publisher = {Bernoulli Society for Mathematical Statistics and Probability},
pages = {2854 -- 2882},
year = {2019},
}

@inproceedings{bernton2018langevin,
  title={Langevin {M}onte {C}arlo and {JKO} splitting},
  author={Bernton, Espen},
  booktitle={Conference on Learning Theory},
  pages={1777--1798},
  year={2018},
  organization={PMLR}
}

@article{pereyra2016proximal,
  title={Proximal {M}arkov chain {M}onte {C}arlo algorithms},
  author={Pereyra, Marcelo},
  journal={Statistics and Computing},
  volume={26},
  number={4},
  pages={745--760},
  year={2016},
  publisher={Springer}
}

@inproceedings{brosse2017sampling,
  title={Sampling from a log-concave distribution with compact support with proximal {L}angevin {M}onte {C}arlo},
  author={Brosse, Nicolas and Durmus, Alain and Moulines, {\'E}ric and Pereyra, Marcelo},
  booktitle={Conference on Learning Theory},
  pages={319--342},
  year={2017},
  organization={PMLR}
}

@article{durmus2018efficient,
  title={Efficient {B}ayesian computation by proximal {M}arkov chain {M}onte {C}arlo: when {L}angevin meets {M}oreau},
  author={Durmus, Alain and Moulines, Eric and Pereyra, Marcelo},
  journal={SIAM Journal on Imaging Sciences},
  volume={11},
  number={1},
  pages={473--506},
  year={2018},
  publisher={SIAM}
}

@article{bubeck2018sampling,
  title={Sampling from a log-concave distribution with projected {L}angevin {M}onte {C}arlo},
  author={Bubeck, S{\'e}bastien and Eldan, Ronen and Lehec, Joseph},
  journal={Discrete \& Computational Geometry},
  volume={59},
  number={4},
  pages={757--783},
  year={2018},
  publisher={Springer}
}

@article{vempala2005geometric,
  title={Geometric random walks: a survey},
  author={Vempala, Santosh},
  journal={Combinatorial and computational geometry},
  volume={52},
  number={573-612},
  pages={2},
  year={2005},
  publisher={Cambridge University Press Cambridge}
}

@article{milman2009role,
  title={On the role of convexity in isoperimetry, spectral gap and concentration},
  author={Milman, Emanuel},
  journal={Inventiones Mathematicae},
  volume={177},
  number={1},
  pages={1--43},
  year={2009},
  publisher={Springer}
}

@article{Chen+26HighAccStoch,
      title={High-accuracy log-concave sampling with stochastic queries}, 
      author={Fan Chen and Sinho Chewi and Constantinos Daskalakis and Alexander Rakhlin},
      year={2026},
      journal={arXiv preprint 2602.14342},
}

@article{Kook25Zeroth,
      title={Zeroth-order log-concave sampling}, 
      author={Yunbum Kook},
      year={2025},
      journal={arXiv preprint 2507.18021},
}

@inproceedings{KooVem25Integration,
author = {Kook, Yunbum and Vempala, Santosh S.},
title = {Sampling and integration of logconcave functions by algorithmic diffusion},
year = {2025},
publisher = {Association for Computing Machinery},
address = {New York, NY, USA},
booktitle = {Proceedings of the 57th Annual ACM Symposium on Theory of Computing},
pages = {924--932},
numpages = {9},
location = {Prague, Czechia},
series = {STOC '25}
}

@InProceedings{Yuan+23Networks,
  title = 	 {On a class of {G}ibbs sampling over networks},
  author =       {Yuan, Bo and Fan, Jiaojiao and Liang, Jiaming and Wibisono, Andre and Chen, Yongxin},
  booktitle = 	 {Proceedings of Thirty Sixth Conference on Learning Theory},
  pages = 	 {5754--5780},
  year = 	 {2023},
  editor = 	 {Neu, Gergely and Rosasco, Lorenzo},
  volume = 	 {195},
  series = 	 {Proceedings of Machine Learning Research},
  month = 	 {7},
  publisher =    {PMLR},
}

@article {BriPed25HeatFlow,
    AUTHOR = {Brigati, Giovanni and Pedrotti, Francesco},
     TITLE = {Heat flow, log-concavity, and {L}ipschitz transport maps},
   JOURNAL = {Electron. Commun. Probab.},
  FJOURNAL = {Electronic Communications in Probability},
    VOLUME = {30},
      YEAR = {2025},
     PAGES = {Paper No. 71, 12},
}

@inproceedings{SalKovRic19ProxLangevin,
 author = {Salim, Adil and Kovalev, Dmitry and Richtarik, Peter},
 booktitle = {Advances in Neural Information Processing Systems},
 editor = {H. Wallach and H. Larochelle and A. Beygelzimer and F. d\textquotesingle Alch\'{e}-Buc and E. Fox and R. Garnett},
 pages = {},
 publisher = {Curran Associates, Inc.},
 title = {Stochastic proximal {L}angevin algorithm: potential splitting and nonasymptotic rates},
 volume = {32},
 year = {2019}
}

@inproceedings{SalRic20ProxLangevin,
 author = {Salim, Adil and Richtarik, Peter},
 booktitle = {Advances in Neural Information Processing Systems},
 editor = {H. Larochelle and M. Ranzato and R. Hadsell and M.F. Balcan and H. Lin},
 pages = {3786--3796},
 publisher = {Curran Associates, Inc.},
 title = {Primal dual interpretation of the proximal stochastic gradient {L}angevin algorithm},
 volume = {33},
 year = {2020}
}

@article {HabHolPoc24Subgrad,
    AUTHOR = {Habring, Andreas and Holler, Martin and Pock, Thomas},
     TITLE = {Subgradient {L}angevin methods for sampling from nonsmooth potentials},
   JOURNAL = {SIAM J. Math. Data Sci.},
  FJOURNAL = {SIAM Journal on Mathematics of Data Science},
    VOLUME = {6},
      YEAR = {2024},
    NUMBER = {4},
     PAGES = {897--925},
}

@article {DurMajMia19LMCCvxOpt,
    AUTHOR = {Durmus, Alain and Majewski, Szymon and Miasojedow, B\l{}a\.{z}ej},
     TITLE = {Analysis of {L}angevin {M}onte {C}arlo via convex optimization},
   JOURNAL = {J. Mach. Learn. Res.},
  FJOURNAL = {Journal of Machine Learning Research (JMLR)},
    VOLUME = {20},
      YEAR = {2019},
     PAGES = {Paper No. 73, 46},
}

@article{Bakry1985,
     author = {Bakry, Dominique and \'Emery, Michel},
     title = {Diffusions hypercontractives},
     journal = {S\'eminaire de probabilit\'es},
     pages = {177--206},
     year = {1985},
     publisher = {Springer - Lecture Notes in Mathematics},
     volume = {19},
     mrnumber = {889476},
     zbl = {0561.60080},
}
    
\end{document}